\pdfoutput=1

\documentclass{siamart190516}
\usepackage{amsmath,amsfonts,amssymb}
\usepackage{graphicx}                          
\usepackage{setspace}  
\usepackage{caption}
\usepackage{listings}
\usepackage{tabularx}
\usepackage{xcolor}
\usepackage{indentfirst}
\usepackage{subcaption}

\usepackage{multirow}
\usepackage{overpic}

\usepackage[normalem]{ulem}

\usepackage{verbatim}

\usepackage{enumerate}

\usepackage{tikz}

\usepackage{bm}

\usepackage[font=footnotesize,labelfont=bf]{caption}

\newtheorem{remark}{Remark}[section] 

\newcommand\dd{\mathrm{d}}
\newcommand\pp{\partial}

\newcommand\x{\bm{x}}
\newcommand\uvec{\mathbf{u}}
\newcommand\X{\mathbf{X}}
\newcommand\y{\bm{y}}

\newcommand\dvec{\bm{d}}

\newcommand\avec{\bm{\xi}}

\begin{document}

\title{A Variational Lagrangian scheme for a phase field model: a discrete energetic variational approach}

\author{Chun Liu\thanks{Department of Applied Mathematics, Illinois Institute of Technology, Chicago, IL 60616, USA (\email{cliu124@iit.edu, ywang487@iit.edu})} \and Yiwei Wang\footnotemark[1]}

\date{}
\maketitle                            

\begin{abstract}
  In this paper, we propose a variational Lagrangian scheme for a modified phase-field model, which can compute the equilibrium states of the original Allen-Cahn type model. Our discretization is based on a prescribed energy-dissipation law in terms of the flow map. By employing a discrete energetic variational approach, this scheme preserves the variational structure of the continuous energy-dissipation law and is energy stable.
  Plentiful numerical tests show that, by choosing the initial value properly, our methods can compute the desired equilibrium state and capture the thin diffuse interface with a small number of mesh points.
\end{abstract}

\section{Introduction}

Phase field models, i.e., diffuse interface models, have been a successful tool in studying many problems arise in physics, biology, material science  and  image processing \cite{chen2002phase, feng2005energetic, chen1998applications, du2009energetic, samson2000variational, bertozzi2012diffuse, brett2014phase}.
Due to the important applications, there is substantial interest in developing efficient numerical methods for phase-field models \cite{liu2003phase, feng2007analysis, feng2006spectral, di2008general, hua2011energy, shen2009efficient, sulman2011optimal, church2019high, xu2019stability, shen2019new, Shen2012modeling}.

From a modeling perspective, phase-field models can be classified into two categories, known as Allen-Cahn type \cite{allen1972ground} and Cahn-Hilliard type \cite{CH58}. The Allen-Cahn type models are typical examples of $L^2-$gradient flows \cite{shen2019new}, while the Cahn-Hilliard type models, which are concerned with a conserved quantity, are examples of $H^1$-diffusions \cite{Giga2017}. Although numerical methods for both types of phase-field models are well developed \cite{liu2003phase, shen2009efficient, shen2019new, du2019phase}, most of them are \emph{Eulerian methods}, which solve the equation of the ``phase'' function $\varphi$ in a fixed grid \cite{shen2009efficient}.
In order to resolve the thin diffuse interface, one must have mesh sizes much smaller than the width of the thin diffuse interface \cite{unverdi1992front, merriman1994motion, du2019phase}, which requires huge computational efforts.
This difficulty is often handled by using adaptive mesh techniques \cite{provatas1998efficient, adler2011first} or moving mesh approaches \cite{feng2006spectral, di2008general, shen2009efficient, sulman2011optimal}.

For many real problems modeled by Allen-Cahn type phase field models, the goal is to find stationary states of the free energy functional.
The purpose of this paper is to propose a variational Lagrangian scheme for a modified phase-field model, which can compute the equilibrium states of the original Allen-Cahn type model. The approach presented here can be extended to other free energy minimization problem.Compared with Eulerian methods, Lagrangian methods, which are often self-adaptive, have potential advantages for problems involving singularity, sharp interface and free boundary. 
Recently, there has been an increasing interest in applying Lagrangian schemes to generalized diffusions, such as the porous medium equation and nonlinear Fokker-Plank equations \cite{carrillo2009numerical, westdickenberg2010variational, carrillo2016numerical, junge2017fully, carrillo2017numerical, matthes2017convergent, carrillo2018lagrangian, liu2019lagrangian}. However, it is more difficult to construct Lagrangian schemes for $L^2$-gradient flows. Unlike generalized diffusions, which have natural variational structures on the Lagrangian maps \cite{evans2005diffeomorphisms, junge2017fully, carrillo2018lagrangian, liu2019lagrangian}, the variational structures of $L^2$-gradient flows are on the physical variables defined in the Eulerian coordinates. Moreover, as a drawback of all Lagrangian methods, the meshes of Lagrangian solutions may become too skew, which not only influence the accuracy of the solution, but also may result in premature termination of Lagrangian calculations \cite{huang2010adaptive}.

In order to overcome these difficulties, we first propose an energy-dissipation law for a phase-field model, given by
\begin{equation}\label{ED1}
\frac{\dd}{\dd t} \int_{\Omega} W(\varphi, \nabla \varphi) \dd \x = - \int_{\Omega} \gamma |\varphi_t|^2 + \nu |\nabla \uvec|^2 \dd \x,
\end{equation}
where $\varphi$ is a ``phase'' function satisfying a transport equation
\begin{equation}\label{Transport1}
\varphi_t + \nabla \varphi \cdot \uvec = 0,
\end{equation}
$\uvec$ is the virtual velocity associated with the Lagrangian map, and $W(\varphi, \nabla \varphi)$ is the free energy density. This model is inspired by phase-field models of mixture of two incompressible fluids \cite{yue2004diffuse, hyon2010energetic, adler2011first}.
For $\nu = 0$, this model employs the same energy-dissipation law of Allen-Cahn type models. So one can view (\ref{ED1}) as a modified Allen-Cahn type model. The additional term in the dissipation part of (\ref{ED1}) can be viewed as a regularization term on Lagrangian maps, which plays an essential role in calculations. The dissipation part imposes a mechanism to minimize the total free energy in terms of the Lagrangian map for the given initial condition $\varphi_0(\X)$.


By employing an energetic variational approach, we can obtain the corresponding PDE of this system, given by
\begin{equation}\label{Eq_u}
  \begin{cases}
    &  -  \nu \Delta \uvec + \gamma \left(\nabla \varphi \otimes \nabla \varphi \right) \uvec = - \nabla \cdot \left( \frac{\pp W}{\pp \nabla \varphi} \otimes  \nabla \varphi  -  W(\varphi, \nabla \varphi) \mathrm{I} \right) \\
    & \varphi_t + \nabla \varphi \cdot \uvec = 0, \\
  \end{cases}
\end{equation}
subject to suitable initial and boundary conditions. Formally, it is straightforward to reformulate (\ref{ED1}) and (\ref{Eq_u}) in terms of a Lagrangian map and its time derivative. 
Hence, based on the energy-dissipation law (\ref{ED1}), we can construct a variational-structure-preserved Lagrangian scheme by employing a discrete energetic variational approach \cite{liu2019lagrangian}.

The rest of this paper is organized as follows. We first give a detailed description to our phase-field model in the next section. Then we construct our variational Lagrangian scheme by a discrete energetic variational approach in Sect. \ref{sec:scheme}. Plentiful numerical tests to validate our methods are shown in Sect. \ref{sec:results}.

\section{Model development}\label{sec:model}
In this section, we give a detailed description of our phase-field model by an energetic variational approach \cite{liu2009introduction, Giga2017}, including the motivation of proposing the energy-dissipation law (\ref{ED1}). 

\subsection{Energetic variational approach}
An energetic variational approach, originated from  pioneering work of Onsager \cite{onsager1931reciprocal, onsager1931reciprocal2} and Rayleigh \cite{strutt1871some} provides a general framework to determine the dynamics of system from a prescribed energy-dissipation law through two distinct variational processes: Least Action Principle (LAP) and Maximum Dissipation Principle (MDP)  \cite{liu2009introduction, Giga2017}. During the last decade, this approach has been successfully applied to build up many mathematical models in physics, chemistry and biochemistry \cite{feng2005energetic, liu2009introduction, sun2009energetic, eisenberg2010energy, Giga2017, wang2020field}.

For an isothermal closed system, an energy-dissipation is given by
\begin{equation}
\frac{\dd}{\dd t} E^{\text{total}}(t) = - 2 \mathcal{D}(t),
\end{equation}
which is the consequence of the first and second laws of thermodynamics \cite{Giga2017}.
Here $E^{\text{total}}$ is the total energy, which is the sum of the Helmholtz free energy $\mathcal{F}$ and the kinetic energy $\mathcal{K}$, and $2\mathcal{D}$ is the rate of energy dissipation. 
The Least Action Principle states that the equation of motion for a Hamiltonian system can be derived from the variation of the action functional  $\mathcal{A}(\x) = \int_0^T \mathcal{K} - \mathcal{F} \dd t$ with respect to the flow maps $\x$ (the trajectory in Lagrangian coordinates) if applicable \cite{arnol2013mathematical, Giga2017}, i.e.,
\begin{equation}
  \delta \mathcal{A} =  \int_{0}^T \int_{\Omega(t)} (f_{\text{inertial}} - f_{\text{conv}})\cdot \delta \x  \dd \x \dd t.
\end{equation}
It gives a unique procedure to derive the conservative forces for the system. On the other hand, for a dissipative system ($\mathcal{D} \geq 0$), the dissipative force can be obtained by minimization of the dissipation functional $\mathcal{D}$ with respect to the ``rate'' $\x_t$ in the regime of linear response \cite{de2013non}, known as Onsager's Maximum Dissipation Principle (MDP), i.e.,
\begin{equation}
\delta \mathcal{D} = \int_{\Omega(t)} f_{\text{diss}} \cdot \delta \x_t~ \dd \x.
\end{equation}
Hence, the force balance condition ($f_{\text{inertial}} = f_{\text{conv}} + f_{\text{diss}}$) results in
\begin{equation}\label{FB}
 \frac{\delta \mathcal{D}}{\delta \x_t} = \frac{\delta \mathcal{A}}{\delta \x},
\end{equation}
which is the dynamics of the system. We refer the reader to \cite{Giga2017} for more detailed descriptions of energetic variational approaches and we only consider systems without kinetic energy, i.e. $\mathcal{K} = 0$, throughout this paper.

\subsection{Energetic variational approaches to phase-field models}
From an energetic variational viewpoint, Allen-Cahn and Cahn-Hilliard type of models, provide a dynamics to minimize the free energy functional
\begin{equation}\label{mix_energy}
\mathcal{F}[\varphi, \nabla \varphi] = \int_{\Omega} W(\varphi, \nabla \varphi) \dd \x,
\end{equation}
for $\varphi$ in some admissible set $\mathcal{H}$ subject to some boundary conditions on $\pp \Omega$. Here,
\begin{equation}
  \varphi =
  \begin{cases}
    & ~1 \quad \text{Phase 1} \\
    & -1 \quad \text{Phase 2}, \\
   \end{cases}
\end{equation}
is a ``phase'' function that introduced to identify the two phases, $W(\varphi, \nabla \varphi)$ is the free energy density given by
\begin{equation}\label{Def_W_phi}
W(\varphi, \nabla \varphi) = \frac{1}{2} |\nabla \varphi|^2 + V(\varphi),
\end{equation}
where $V(\varphi)$ is the interfacial (potential) energy that often taken as a double-well potential
\begin{equation}
 V(\varphi) = \frac{1}{4 \epsilon^2} (\varphi^2 - 1)^2.
\end{equation}
Different phase-field models can be derived by different choices of admissible sets $\mathcal{H}$ and dissipation functional $2 \mathcal{D}$.

In Allen-Cahn type models,  $\mathcal{H}$ is often chosen to be $H^1(\Omega)$ with a suitable boundary conditions, and the energy-dissipation law is given by
\begin{equation}\label{ED_GD}
\frac{\dd}{\dd t} \mathcal{F}(\varphi, \nabla \varphi) = - \int_{\Omega} \frac{1}{\gamma} |\varphi_t|^2 \dd \x,
\end{equation}
where $\gamma > 0$ is the dissipation rate \cite{Giga2017}. We take $\gamma = 1$ in the following.
The energy-dissipation law (\ref{ED_GD}) can be viewed as a gradient flow of the phase function $\varphi(\x, t)$, which specify the dynamics approaching to the equilibrium of system.  According to the general framework of an energetic variational approach, the corresponding gradient flow equation can be derived by first performing LAP with respect to $\varphi$ and MDP with respect to $\varphi_t$:
\begin{equation}
  \begin{aligned}
    & \text{LAP}: \frac{\delta \mathcal{A}}{\delta \varphi} = - \frac{\delta \mathcal{F}}{\delta \varphi} =  \nabla \cdot \frac{\pp  W}{\pp \nabla \varphi} - \frac{\pp W}{\pp \varphi}, \\
    & \text{MDP}: \frac{\delta \frac{1}{2} \mathcal{D}}{\delta \varphi_t} = \varphi_t, \\
  \end{aligned}
\end{equation}
where we assume all boundary terms vanish due to the given boundary condition. Then the force balance equation (\ref{FB}) leads to an Allen-Cahn type equation
\begin{equation}
  \varphi_t = \nabla \cdot \frac{\pp W}{\pp \nabla \varphi} - \frac{\pp W}{\pp \varphi}.
\end{equation}
A stationary solution of the Allen-Cahn type equation satisfies the Euler-Lagrangian equation of the functional (\ref{mix_energy}), i.e.,
\begin{equation}\label{EL_Sol}
\frac{\delta W}{\delta \varphi} - \nabla \cdot \left(\frac{\delta W}{\delta \nabla \varphi} \right) = 0.
\end{equation}

The above derivation performs an energetic variational approach in terms of $\varphi$ and $\varphi_t$. We call this as an {\bf Eulerian approach}, in which $\varphi$ can be viewed as a generalized coordinate of the system \cite{doi2011onsager}. There is an alternative way to derive a dynamic of the system, known as the {\bf Lagrangian approach} \cite{Giga2017}.
Instead of studying the evolution of phase function $\varphi(\x, t)$ directly, the Lagrangian approach study the evolution of a \emph{Lagrangian map}, or \emph{flow map}, $\x(\X, t)$  for given initial condition $\varphi_0(\X)$. For fixed $t$, $\x^t(\X) = \x(\X, t)$ is a diffeomorphism between the initial domain $\Omega^0$ and the current domain $\Omega^t$, known as a \emph{deformation map} \cite{temam2005mathematical, gonzalez2008first}. For fixed $\X$, $\x(\X, t)$ is the trajectory of the particle labeled by $\X$. We can view $\X \in \Omega^0$ are Lagrangian coordinates and $\x \in \Omega^t$ are Eulerian coordinates.

For a given flow map $\x(\X, t)$, we can define the virtual velocity in Eulerian coordinate, $\uvec(\x(\X, t), t)$ as
\begin{equation}
\uvec(\x(\X, t), t) = \x_t(\X, t).
\end{equation}
Another important quantity associated with $\x(\X, t)$ is the deformation tensor $F(\X, t)$, defined by 
 \begin{equation}
   F(\X, t) = \nabla_{\X} \x(\X, t),
 \end{equation} 
which carries all the information about how the physical quantity $\varphi$ transport with the flow. 
Since $\x(\X, t)$ is a one-to-one map between $\Omega^0$ and $\Omega^t$ for fixed $t$, we can enforce $\det F(\X, t) > 0$, which means the map $\x(\X, t)$ is orientation-preserving for $\forall t$ \cite{gonzalez2008first}.


In order to get the equation of $\x(\X, t)$, we shall impose the kinematic relation to the physical quantity $\varphi$. Then the dynamics of $\varphi(\x(\X, t), t)$ will be totally determined by the dynamics of the flow map $\x(\X, t)$.
For Allen-Cahn type models, it is often assumed that $\varphi$ satisfies 
\begin{equation}\label{Transport}
\varphi(\x(\X, t), t) = \varphi_0(\X),
\end{equation}
where $\varphi_0(\X)$ is the initial condition. 
One can view (\ref{Transport}) as a composition between $\varphi_0$ and inverse flow map $\X^{-1}(\x, t)$ at time $t$, that is
\begin{equation}\label{composition}
  \varphi(\x, t) = \varphi_0 \circ \X^{-1}(\x, t).
\end{equation}
From the kinematic equation (\ref{Transport}), we have
\begin{equation}
0 = \frac{\dd}{\dd t} \varphi(\x(\X, t), t) = \pp_t \varphi + \nabla \varphi \cdot \uvec.
\end{equation}
Hence, $\varphi(\x, t)$ satisfies scalar transport equation
\begin{equation}\label{ST}
\pp_t \varphi + \uvec \cdot \nabla \varphi = 0,
\end{equation}
in Eulerian coordinates.

\begin{remark}
  The above transport relation (\ref{ST}) is the macroscopic transport on the microscopic variable $\varphi$, which might only be valid locally. 
  The complicated phase evolution, such as interface merging or pinching off, which is a consequence of microscopic evolution of $\varphi$, cannot be described by this kinematic.  
\end{remark} 

Within the kinematic (\ref{Transport}), $\varphi(\x)$ is determined by $\x(\X, t)$ for given $\varphi_0(\X)$. Hence, we can propose a energy-dissipation law in terms of $\x(\X, t)$ and $\x_t(\X, t)$ to characterize the dynamics of the flow map, that is
\begin{equation}\label{ED_Lag}
\frac{\dd}{\dd t} \mathcal{F} [\x] = - 2 \mathcal{D}[\x, \x_t],
\end{equation}
where
\begin{equation}
\mathcal{F}[\x] = \int_{\Omega_0} W(\varphi_0, F^{-\rm T} \nabla_{\X} \varphi_0) \det F \dd \X,
\end{equation}
and $2 \mathcal{D}[\x, \x_t] \geq 0$ is the rate of energy dissipation. One can view the free energy as a function of $\X$, $\x(\X, t)$ and $F$, denoted by
\begin{equation}\label{Energy_F}
  \mathcal{F}[\x]  = \int_{\Omega_0} \mathcal{W}(\X, F) \dd \X.
\end{equation}

The energy-dissipation law (\ref{ED_Lag}) can be viewed as a generalized gradient flow of the flow map $\x(\X, t)$.
Since we are only concerned with equilibria of the system, the choice of dissipation only effects the dynamics approaching to equilibria. We'll discuss this later.

The evolution equation of the flow map $\x(\X, t)$ can be derived by employing an energetic variational approach, that is
\begin{equation}\label{Eq_flow}
\frac{\delta \mathcal{D}}{\delta \x_t} = - \frac{\delta \mathcal{F}}{\delta \x},
\end{equation}
where [See Appendix. A for the detailed computation]
\begin{equation}\label{F_x}
\frac{\delta \mathcal{F}}{\delta \x} = \nabla \cdot \left( \frac{\pp W}{\pp \nabla \varphi} \otimes  \nabla \varphi  -  W(\varphi, \nabla \varphi) \mathrm{I} \right).
\end{equation}
A stationary solution in the Lagrangian approach satisfies
\begin{equation}\label{sol_Lag}
 \nabla \cdot \left( \frac{\pp W}{\pp \nabla \varphi} \otimes  \nabla \varphi  -  W(\varphi, \nabla \varphi) \mathrm{I} \right)  = 0.
\end{equation}

\begin{remark}
  The Lagrangian approach minimizes the free energy functional in the admissible set
  \begin{equation}
    \mathcal{Q} = \{ \varphi(\x) ~|~ \varphi(\x) = \varphi_0 \circ \X^{-1}(\x), ~ \X^{-1}: \Omega \rightarrow \Omega^0  \},
  \end{equation}
  which is different with that in the Eulerian approach. So it is subtle to choose a suitable $\varphi_0$ to get a desired equilibrium. For the classical Allen-Cahn equation, since $|\varphi| \leq 1$, it is not difficult to choose a proper $\varphi_0$ . In general, $\varphi_0$ can be obtained by some Eulerian approach. We can also update $\varphi_0$ during the evolution of the flow map. 
\end{remark}

\begin{remark}
  If $\varphi$ is a conserved quantity that satisfies
  \begin{equation}
   \frac{\dd}{\dd t} \int_{\Omega} \varphi(\x, t) \dd \x = 0,
  \end{equation}
  then the kinematic equation is given by
  \begin{equation}\label{kinemtic_diffusion}
    \varphi(\x(\X, t), t) = \frac{\varphi_0(\X)}{\det F}.
  \end{equation}
  This is the kinematic for the Cahn-Hilliard type equation, which can be viewed as a generalized diffusion with the energy-dissipation law given by \cite{liu2019energetic}
  \begin{equation}
    \frac{\dd}{\dd t} \int_{\Omega} W(\varphi, \nabla \varphi) \dd \x = - \int_{\Omega} \varphi^2 |\uvec|^2 \dd \x.
  \end{equation}
   Both Allen-Cahn and Cahn-Hilliard equations types are driven by the same mixture energy (\ref{mix_energy}), but the kinematic and dissipation mechanisms are different.
\end{remark}

 Although the equations for the stationary solutions obtained by the Eulerian approach (variation on the phase variable $\varphi(\x, t)$) and the Lagrangian approach (variation on the flow maps) look different [(\ref{EL_Sol}) and (\ref{sol_Lag})], formally one can easily show that \cite{liu2003phase}: 
\begin{theorem}\label{Theorem1}
  For a given energy functional (\ref{mix_energy}), all smooth (regular enough) solutions of the Euler-Lagrangian equation:
  \begin{equation}\label{eq_EL}
    - \nabla \cdot \left( \frac{\pp W}{\pp \nabla \varphi}  \right) + \frac{\pp W}{\pp \varphi} = 0
  \end{equation}
  also satisfy the equation
  \begin{equation}\label{eq_Lg}
    \nabla \cdot \left( \frac{\pp W}{\pp \nabla \varphi} \otimes \nabla \varphi - W \mathrm{I} \right) = 0.
  \end{equation}
\end{theorem}
This result indicates connection between variation with respect to $\varphi$ and the variation with respect to flow map through Legendre transform \cite{liu2003phase}. In general, the weak solution of the Euler-Lagrange equation (\ref{eq_EL}) may not satisfy (\ref{eq_Lg}). In the theory of harmonic map, a weak solution of (\ref{eq_EL}) that also satisfies the weak form of (\ref{eq_Lg}) is known as a \emph{stationary weak solution} \cite{liu2003phase, schoen1982regularity, ball2017mathematics}. From a numerical perspective, this theorem indicates that all equilibria in the Eulerian approach can be obtained from the Lagrangian approach with a proper choice $\varphi_0(\X)$.  However, for a given $\varphi_0(\X)$, the Lagrangian calculation may not end up with the same equilibrium of the Eulerian approaches.

\subsection{Dissipation functional}
In this subsection, we discuss the choice of dissipation functional in Lagrangian approaches to phase-field models. Different choices of dissipation provide different dynamics approaching equilibria of the system. Since we may have multiple equilibria for the free energy like (\ref{mix_energy}) \cite{Yin2020prl}, different dynamics may end up with different equilibria for given $\varphi_0(\X)$. 

By using the kinematic relation (\ref{composition}) and (\ref{ST}), the dissipation for the gradient flow  (\ref{ED_GD}) can be reformulated in terms of $\x(\X, t)$ and $\x_t(\X, t)$, that is
\begin{equation}\label{ED_L1}
\mathcal{D}[\x, \x_t] = - \frac{1}{2} \int |\nabla \varphi \cdot \uvec|^2 \dd \x
\end{equation}
for  given initial condition $\varphi_0(\X)$. The equation of the flow map $\x(\X, t)$ can be obtained via a standard energetic variational approach {\color{blue} (\ref{Eq_flow})}, which is
\begin{equation}\label{FD_1}
(\nabla \varphi \otimes \nabla \varphi) \uvec =  - \nabla \cdot \left( \frac{\pp W}{\pp \nabla \varphi} \otimes  \nabla \varphi  - W(\varphi, \nabla \varphi) \mathrm{I}   \right).
\end{equation}
Here the right hand side is obtained by the LAP, which corresponds to $\frac{\delta \mathcal{A}}{\delta \x}$ [see (\ref{F_x})], while the left hand side is obtained by MDP, i.e. $\frac{\delta \mathcal{D}}{\delta \uvec} = (\nabla \varphi \otimes \nabla \varphi) \uvec$.
In a recent work \cite{cheng2019new}, the authors study numerical methods for equation (\ref{FD_1}) in one-dimension by discretizing $\x(\X, t)$ directly. Their results show that the dynamics (\ref{FD_1}) can capture the thin diffuse interfaces of Allen-Cahn type equations with a small number of mesh points in 1D. However, 
the energy-dissipation law (\ref{ED_L1}) may not be suitable for Lagrangian calculations, especially for high dimensions $d \geq 2$. Indeed, since $\nabla \varphi \otimes \nabla \varphi$ is a rank one matrix, $\nabla \varphi \otimes \nabla \varphi$ is not a invertible matrix for $d \geq 2$, so $\uvec$ is not well-defined everywhere. Moreover, even for one-dimensional cases,  $\nabla \varphi \otimes \nabla \varphi$ is almost zero in non-interfacial regions, which restricts the choice of $\varphi_0$.

The degeneracy of the equation (\ref{FD_1})  motivates us to consider a different dissipation functional by adding a new term, that is 
\begin{equation}\label{New_dissipation}
\mathcal{D} = \frac{1}{2} \int_{\Omega} |\nabla \varphi \cdot \uvec|^2 +  \nu | \nabla \uvec|^2 \dd \x,
\end{equation}
where $\nu$ is a constant.
By a direct computation, for such an energy-dissipation law, the dynamics of the system is given by
\begin{equation}\label{eq_1}
    -  \nu \Delta \uvec + \left(\nabla \varphi \otimes \nabla \varphi \right) \uvec = - \nabla \cdot \left( \frac{\pp W}{\pp \nabla \varphi} \otimes  \nabla \varphi  -  W(\varphi, \nabla \varphi) \mathrm{I} \right), \\
\end{equation}
which gives us the equation of the flow map $\x(\X, t)$ in Lagrangian coordinates. The energy-dissipation law (\ref{New_dissipation}) fixes the degeneracy of $\nabla \varphi \otimes \nabla \varphi$. Moreover, from a computational perspective, 
$\nu |\nabla \uvec|^2$ can be viewed as a regularization term to the flow map $\x(\X, t)$, which controls the quality of mesh generated by the flow map.

In Lagrangian coordinates, (\ref{New_dissipation}) can be written as
\begin{equation}\label{New_dissipation_1}
  \mathcal{D}[\x, \x_t]  = \frac{1}{2} \int_{\Omega^0}  \Big| (F^{-\rm T} \nabla_{\X} \varphi_0) \cdot \x_t \Big|^2   +  \nu  |\nabla_{\X} \x_t F^{-1}|^2 \det F \dd \X.
\end{equation}
In order to simplify the numerical implementation, we replace $ |\nabla_{\X} \x_t F^{-1}|^2$ by $|\nabla_{\X} \x_t|$ in the following. Then the equation for flow map $\x(\X, t)$ is given by (recall $\frac{\pp W}{\pp \nabla \varphi} = \nabla \varphi$ due to (\ref{Def_W_phi}))
\begin{equation}\label{eq_Lag_X}
\begin{aligned}
  & - \nu \Delta_{\X} \uvec + \left( (F^{-\rm T} \nabla_{\X} \varphi_0) \otimes (F^{-\rm T} \nabla_{\X} \varphi_0) \right) \uvec \\
  & \quad = - \nabla_{\X} \left( (F^{-\rm T} \nabla_{\X} \varphi_0) \otimes (F^{-\rm T} \nabla_{\X} \varphi_0) - W(\varphi_0, F^{-\rm T} \nabla_{\X} \varphi_0) {\rm I} \right) : F^{-1}, \\
\end{aligned}
\end{equation}
subject to the initial condition $\x(\X, 0) = \X$ and a suitable boundary condition, where $A : B = \sum_{j, k = 1}^n A_{ijk} B_{jk}$ for $A \in \mathbb{R}^{n \times n \times n }$ and $B \in \mathbb{R}^{n \times n}$.

\begin{remark}
  It is worth mentioning that the additional terms in both (\ref{New_dissipation}) and (\ref{New_dissipation_1}) are not physically acceptable viscosity for compressible fluids, and we add them into the dissipation functional only for the numerical purpose. 
  More specifically, let
  \begin{equation}
    \x^{*}(\X, t) = R(t) \x(\X, t),
  \end{equation}
  then according to the frame-indifference, we should have
  \begin{equation}
    \mathcal{D}(\x, \x_t) =  \mathcal{D}(\x^{*}, \x_t^{*}).
  \end{equation}
  Note
  \begin{equation*}
    \nabla_{\x^{*}} \uvec^{*} = \dot{R} R^{-1} + R \nabla_{\x} \uvec R^{-1}, \quad \nabla_{\X} \uvec^{*} = \dot{R} F + R \nabla_{\X} \uvec,
  \end{equation*}
  hence, it is easy to show that the additional terms in both (\ref{New_dissipation}) and (\ref{New_dissipation_1})  conflict with the frame-indifference. 
  For compressible flow, a physically acceptable viscosity in the dissipation is often taken as
  \begin{equation}
    \mathcal{D} = \int \nu \Big| \frac{1}{2} (\nabla \uvec + \left( \nabla \uvec)^{\rm T}) \Big|^2 + (\zeta - \frac{2}{3} \nu \right) |\nabla \cdot \uvec|^2 \dd \x,
  \end{equation}
  where $\nu > 0$ and $\eta > 0$. We refer the reader to \cite{antman1998physically, dafermos2005hyperbolic} for more detailed discussions. 
\end{remark}

At the end of this section, We should emphasize that the above derivation is rather formal, in which we assume that the flow map exists at least locally. The goal of this paper is designing some Lagrangian schemes that preserve the above variational structures in a discrete level.
More analysis are certainly need to show the existence of flow map. We refer the interested reader to \cite{evans2005diffeomorphisms,dacorogna1981relaxation} for some theoretical results on some related but different systems. 

\section{Numerical Scheme}\label{sec:scheme}

In this section, we construct our variational Lagrangian scheme for the phase-field model with the energy-dissipation law (\ref{ED1}) by a discrete energetic variational approach \cite{liu2019lagrangian}. Instead of considering a particular weak form of the flow map equation (\ref{eq_Lag_X}), a discrete energetic variational approach, 
which performs an energetic variational approach in a semi-discrete level, derives a ``semi-discrete equation'' that preserves the variational structure from a discrete energy-dissipation law directly. By introducing a proper temporal discretization to the ``semi-discrete equation'', we can construct an energy stable Lagrangian scheme to our phase-field model.

\subsection{A discrete energetic variational approach}
In general, for a system without kinetic energy, a discrete energy-dissipation law can be written as
\begin{equation}\label{d_EL}
\frac{\dd}{\dd t} \mathcal{F}_h(\bm{\Xi}(t)) = - 2 \mathcal{D}_h(\bm{\Xi}(t), \bm{\Xi}'(t)),
\end{equation}
where $\bm{\Xi}(t) \in \mathbb{R}^{K}$ is the ``discrete'' state variable, $\mathcal{F}_h(\bm{\Xi}(t))$ is the discrete free energy and $ 2 \mathcal{D}_h(\bm{\Xi}(t), \bm{\Xi}'(t))$ is the discrete dissipation.
One can obtain a discrete energy-dissipation law (\ref{d_EL}) from the continuous energy-dissipation law by either discretizing the physical quantity $\varphi(\x, t)$ (Eulerian approach) or the flow map $\x(\X, t)$ (Lagrangian approaches) in space. 

Similar to an energetic variational approach in a continuous level, the governing equation of $\bm{\Xi}(t)$, a system of nonlinear ODEs, can be obtained from the force balance equation
\begin{equation}\label{discret_D_1}
  \frac{\delta  \mathcal{D}_h}{\delta \bm{\Xi}'}(\bm{\Xi}(t), \bm{\Xi}'(t)) = - \frac{\delta \mathcal{F}_h}{\delta \bm{\Xi}} (\bm{\Xi}(t)),
\end{equation}
where the right-hand side comes by performing LAP, taking variation of the discrete action functional
$\mathcal{A}_h(\bm{\Xi}(t)) = \int_{0}^T - \mathcal{F}_h(\bm{\Xi}(t)) \dd t$ with respect to $\bm{\Xi}(t)$, while the left-hand side comes by performing MDP, taking variation of the discrete dissipation functional
$\mathcal{D}_h (\bm{\Xi}(t), \bm{\Xi}^{\prime}(t))$ with respect to $\bm{\Xi}^{\prime}(t)$.

A discrete energetic variational approach follows the strategy of \emph{``discrete-then-variation''}, which has been a powerful tool to construct numerical schemes for complicated systems with variational structures \cite{furihata2010discrete, christiansen2011topics, carrillo2009numerical, carrillo2016numerical, xu2016variational, liu2019lagrangian}. Compared with the traditional \emph{``variation-then-discrete''} approach, the ``semi-discrete'' equation obtained by a discrete energetic variational approach can automatically inherit the variational structure from the continuous level.
One may obtain the same ``semi-discrete'' equation through a ``variation-then-discrete'' approach by choosing a particular weak form for the PDE.


For our phase-field model, in order to get a discrete energy-dissipation law, we first introduce a piecewise linear approximation to the flow map $\x(\X, t)$, 
which can be constructed by a standard finite element method.
In the following, we only discuss the two-dimensional case, the procedure can be easily extended to other spatial dimensions. Let $\mathcal{T}_h$ be a triangulation of domain $\Omega^0$, consists of a set of simplexes $\{ \tau_e ~|~ e = 1,\ldots M \}$ and a set of nodal points $\mathcal{N}_h = \{\X_1, \X_2, \ldots, \X_N \}$. Then the approximated flow map is given by
\begin{equation}\label{x_h}
\x_h(\X, t) = \sum_{i = 1}^{N} \avec_i(t) \psi_i(\X) ~\in~ V_h,
\end{equation}
where
\begin{equation*}
V_h = \{  v \in C(\Omega) ~|~ v~\text{is linear on each element}~\tau_e \in \mathcal{T}_h  \},
\end{equation*}
and $\psi_i(\X) : \mathbb{R}^2 \rightarrow \mathbb{R}$ is the hat function satisfying $\psi_i(\X_j) = \delta_{ij}$.
Since $\x_h(\X_i, t) =  \avec_i(t)$, $\avec_i(t) = (\xi_{i, x}, \xi_{i, y}) \in \mathbb{R}^2$ can be viewed as the coordinate of $i$-th mesh point at $\Omega^t$, and $\avec'_i(t)$ defines the velocity of $i$-the mesh point. Within the above spatial discretization, the discrete state variable of the system $\bm{\Xi}(t)$ is defined by
\begin{equation}
\bm{\Xi}(t) = \left( \xi_{1, x}, \xi_{2, x}, \ldots, \xi_{N,x}, \xi_{1, y}, \xi_{2, y}, \ldots, \xi_{N, y} \right) \in \mathbb{R}^{K},
\end{equation}
where $K = 2N$. For simplicity's sake, we consider the natural boundary condition for the flow map through this section. For the Dirichlet boundary condition that considered in the next section, if $\varphi_0$ is chosen to satisfy the Dirichlet boundary condition, we can set $x(\X, t) = \X$ for $\X \in \pp \Omega_0$, i.e., the velocity of the mesh points on the boundary to be zero, such that the Dirichlet boundary condition is satisfied for $\varphi(\x, t)$.  


The framework of finite element discretization enables us to compute the deformation matrix $F$ explicitly on each element [see the Appendix in \cite{liu2019lagrangian} for the explicit form].  
We denote the deformation matrix $F_h(\x_h(\X, t), t)$ on each element $\tau_e$, which is a constant matrix for fixed $t$, by $F_e(\bm{\Xi}(t)) = \nabla_{\X} |_{\X \in \tau_e} \x_h$. The admissible set of $\bm{\Xi}(t)$ is defined by
\begin{equation}
\mathcal{S}_{ad}^h = \left\{ \bm{\Xi}(t) \in \mathbb{R}^{K} ~|~ \det F_e(\bm{\Xi}(t)) > 0,~\forall e    \right\}.
\end{equation}
It can be noticed that $\mathcal{S}_{ad}^h$ is not a convex set, which imposes difficulties in both simulations and numerical analysis.

Inserting (\ref{x_h}) into the original energy-dissipation law, we can obtain the discrete free energy
\begin{equation}\label{discret_E_HD}
\begin{aligned}
  \mathcal{F}_h(\bm{\Xi}(t)) 
  = \sum_{e = 1}^{M} \int_{\tau_e}  W \left( \varphi_0(X), F_e^{-\rm{T}} \nabla_{X} \varphi_0 \right) \det F_e \dd X,
\end{aligned}
\end{equation}
and the discrete dissipation functional
\begin{equation}
  \begin{aligned}
    \mathcal{D} (\bm{\Xi}(t), \bm{\Xi}'(t)) &  =  \frac{1}{2}  \sum_{e = 1}^{M} \int_{\tau_e}  \left( |(F^{-\rm T}_e \nabla_{\X} \varphi_0) \cdot \uvec_h|^2 + \nu | \nabla_{\X} \uvec_h|^2 \right) \det F_e \dd \X,  \\
  \end{aligned}
\end{equation}
where
\begin{equation}
\uvec_h(\X, t) = \sum_{j=1}^M \bm{\xi}_j'(t) \psi_j(\bm{X}).
\end{equation}
Then by a discrete energetic variation approach, we can derive a system of ordinary differential equations of $\bm{\Xi}(t)$, that is
\begin{equation}\label{Semi_eq_1}
  {\sf D}_h(\bm{\Xi}(t)) \bm{\Xi}'(t) = - \frac{\delta \mathcal{F}_h}{\delta \bm{\Xi}} (\bm{\Xi} (t)).
\end{equation}
We refer readers to the Appendix for the detailed computation of ${\sf D}(\bm{\Xi}(t))$ and the $\frac{\delta \mathcal{F}_h}{\delta \bm{\Xi}}$. Although the explicit forms of both ${\sf D}(\bm{\Xi}(t))$ and $\frac{\delta \mathcal{F}}{\delta \bm{\Xi}}$ may not be available in a general mesh, both of them are easy to obtain during the numerical implementation by summing the results on each element over the mesh.
It can be noticed that ${\sf D}(\bm{\Xi}(t)) \in \mathbb{R}^{K \times K}$ given by
\begin{equation}
{\sf D}(\bm{\Xi}(t)) = {\sf M}(\bm{\Xi}(t)) + \nu {\sf K}(\bm{\Xi(t)})
\end{equation}
with
\begin{equation*}
  {\sf M} =
  \begin{pmatrix}
    {\sf M}_{xx} &  {\sf M}_{xy} \\
    {\sf M}_{yx} &  {\sf M}_{yy} \\
  \end{pmatrix},
  \quad
    {\sf K} =
  \begin{pmatrix}
    {\sf K_0} &  {\sf 0} \\
     {\sf 0} &  {\sf K_0} \\
  \end{pmatrix}.
\end{equation*}
Here ${\sf M}_{\alpha \beta}$ ($\alpha, \beta = x, y$) is the modified mass matrix defined by
\begin{equation}\label{def_M}
{\sf M_{\alpha \beta}}(i, j) = \sum_{e \in N(i)} \pp_{\alpha} \varphi(\x_e) \pp_{\beta} \varphi (\x_e) \det F_e  \int_{\tau_e}\psi_i \psi_j    \dd \X,
\end{equation}
where $\x_e$ is the centroid of $\x^t(\tau_e)$, and ${\sf K_0}$ is the modified stiff matrix defined by
\begin{equation}\label{def_K}
{\sf K_0}(i, j) = \sum_{e \in N(i)} \det F_e \int_{\tau_e} \nabla_{\X} \psi_i \cdot  \nabla_{\X} \psi_j   \dd \X.
\end{equation}
It is easy to show that $\det {\sf M}(\bm{\Xi}) = 0$, ${\sf M}$ is positive semi-definite and ${\sf K}(\bm{\Xi})$ is a positive-define matrix if $\bm{\Xi} \in \mathcal{S}_{ad}^h$ [see Appendix \ref{App_B}]. Hence, the presence of $\nu {\sf K}$ ensures that ${\sf D}$ is positive-definite.


\subsection{Temporal discretization}
Now we discuss the temporal discretization. A numerical scheme can be obtained by introducing a suitable temporal discretization to the ``semi-discrete equation'' (\ref{Semi_eq_1}). An advantage of existing a variational structure in the semi-discrete level is that various of classical numerical schemes can be reformulated as optimization problems \cite{du2019phase, xu2019stability, matthes2019variational}. 
In the current study, we use implicit Euler for temporal discretization. It is not difficult to apply high-order temporal discretization, such as BDF2 or Crank-Nicolson \cite{du2019phase} to our system, which will be studied in the future work.


For given $\bm{\Xi}^n \in \mathcal{S}_{ad}^h$, the implicit Euler scheme for (\ref{Semi_eq_1}) is given by 
\begin{equation}\label{Implicit}
{\sf D}^{n}_{*} \frac{\bm{\Xi}^{n+1} - \bm{\Xi}^n}{\tau} = - \frac{\delta \mathcal{F}_h}{\delta \bm{\Xi}}(\bm{\Xi}^{n+1}).
\end{equation}
where ${\sf D}^n_{*}$ is chosen to be independent with $\bm{\Xi}^{n+1}$, that is taking $\pp_{\alpha} \varphi$, $\pp_{\beta} \varphi$ and $\det F_e$ in (\ref{def_M}) and (\ref{def_K}) to be value at $n-$th step.

Although (\ref{Implicit}) is a system of highly nonlinear equations that is often difficult to solve, by virtue of the variational structures in the semi-discrete level, we can reformulate (\ref{Implicit}) into an optimization problem, given by
\begin{equation}\label{Min_Problem}
\bm{\Xi} = \mathrm{argmin}_{\bm{\Xi} \in \mathcal{S}_{ad}^h} J_n(\bm{\Xi}),
\end{equation}
where
\begin{equation}\label{define_Jn}
  J_n(\bm{\Xi}) = \frac{ \left( {\sf D}^{n}_{*}(\bm{\Xi} - \bm{\Xi}^{n}), (\bm{\Xi} - \bm{\Xi}^{n}) \right)}{2 \tau} + \mathcal{F}_h (\bm{\Xi}).
\end{equation}

There are various of advantages in solving optimization problem (\ref{Min_Problem}) instead of solving the original nonlinear system (\ref{Implicit}) directly. Since $J_n(\bm{\Xi})$ might not be a convex function, solving (\ref{Implicit}) with standard nonlinear solvers, such as fixed-point iterations or Newton-type methods, may only obtain a saddle point or a local minimizer of $J_n(\bm{\Xi})$, which may not decrease the discrete energy. Moreover, the standard nonlinear solver can not guarantee  that the obtained solution is in the admissible set $\mathcal{S}_{ad}^h$.
For the optimization problem (\ref{Min_Problem}), we can use some line-search based optimization method and manually set $$J_n(\bm{\Xi}) = \infty, \quad \bm{\Xi} \notin \mathcal{S}_{ad}^h.$$  Then the line search based method can guarantee  $\det F_e > 0$ as if $J_n(\bm{\Xi}^{n+1}) \leq J_n(\bm{\Xi}^n)$, even though the exact global minimizer of $J_n(\bm{\Xi})$ may not be found. 
Noticed that $J(\bm{\Xi}^{n+1}) \leq J(\bm{\Xi}^{n})$ indicates
\begin{equation}
  \frac{1}{2 \tau} {\sf D}^*_n (\bm{\Xi}^{n+1} - \bm{\Xi^{n}}) \cdot (\bm{\Xi}^{n+1} - \bm{\Xi^{n}}) + \mathcal{F}_h(\bm{\Xi}^{n+1}) \leq \mathcal{F}_h (\bm{\Xi}^{n}).
  \end{equation}
  Hence, our scheme is energy stable satisfies the discrete energy-dissipation law 
  \begin{equation}\label{d_energy_dis_our_scheme}
    \frac{\mathcal{F}_h(\bm{\Xi}^{n+1}) - \mathcal{F}_h(\bm{\Xi}^n)}{\tau} \leq  - \frac{1}{2 \tau^2} {\sf D}_n^{*} (\bm{\Xi}^{n} - \bm{\Xi}^{n+1}) \cdot (\bm{\Xi}^{n} - \bm{\Xi}^{n+1}) \leq 0,
  \end{equation}


{{\color{blue}} For particular form of free energy}, following \cite{carrillo2018lagrangian}, we can prove the existence of a minimal solution of the optimization problem (\ref{Min_Problem}):
\begin{proposition}\label{Thm1}
  For a given initial condition $\varphi_0(\X)$, if the free energy density $\mathcal{W}(\X, F)$ (see (\ref{Energy_F}) for the definition) satisfies $\mathcal{W}(\X, F)  > 0$ for $\det F > 0$ and 
  \begin{equation}\label{Assumption_WF}
    \mathcal{W}(\X, F) \rightarrow 0, \quad \det F \rightarrow \infty,
  \end{equation}
then for given $\bm{\Xi}^n \in F_{ad}^{\bm{\Xi}}$, there exists a solution $\bm{\Xi}^{n+1}$ to numerical scheme (\ref{Implicit}) such that the following discrete energy dissipation law holds, i.e.,
\end{proposition}

\begin{proof}

Due to the assumption (\ref{Assumption_WF}), we know for $\forall \bm{\Xi} \in \pp \mathcal{S}_{ad}^{\bm{\Xi}}$, $J_n(\bm{\Xi}) = \infty$. Following the proof in the Lemma 3.1 in \cite{carrillo2018lagrangian}, the existence of a minimizer can be obtained by showing the set
\begin{equation}
\mathcal{A} = \{ \bm{\Xi} \in \mathcal{S}_{ad}^h ~|~ J_n(\bm{\Xi}) \leq \mathcal{F}_h(\bm{\Xi}^n)  \}
\end{equation}
is a non-empty compact subset of $\mathbb{R}^K$. Obviously, $\bm{\Xi}^n \in \mathcal{A}$, so $\mathcal{A}$ is non-empty. On the other hand, since ${\sf D}^*_n$ is positive-definite, there exists $\lambda_1 > 0$ such that $\forall \bm{\Xi} \in \mathcal{S}_{ad}^g$
\begin{equation}
\| \bm{\Xi} - \bm{\Xi}^n \|^2 \leq \frac{1}{\lambda_1}  {\sf D}^*_n (\bm{\Xi} - \bm{\Xi^{n}}) \cdot (\bm{\Xi} - \bm{\Xi^{n}}) \leq  \frac{2 \tau}{\lambda_1} \left(\mathcal{F}_h(\bm{\Xi}^n) - \mathcal{F}_h(\bm{\Xi}) \right),
\end{equation}
which indicates $\mathcal{A}$ is bounded. So we only need to show $\mathcal{A}$ is closed in $\mathbb{R}^{K}$. For any converged sequence $\{ \bm{\Xi}^{(k)} \}_{k=1}^{\infty} \subset \mathcal{S}_{ad}^h$, our goal is to show that the limit $\widetilde{\bm{\Xi}}$ is in $\mathcal{S}_{ad}^h$. Note for $\forall e \in \{ 1, 2, \ldots M \}$ and all $k$,
\begin{equation}
  \begin{aligned}
    \mathcal{F}_h(\bm{\Xi}^n) \geq \mathcal{F}_h(\bm{\Xi}^{(k)}) \geq \mathcal{W}(\X_e, F_e) |\tau_e|
  \end{aligned}
\end{equation}
where $|\tau_e|$ is the area of element $\tau_e$. 
Since $\mathcal{W}(X, F) \rightarrow \infty$ if $\det F \rightarrow 0$, we can conclude that $\det F_e(\bm{\Xi}^{(k)}) > 0$ is uniformly bounded away from zero. So $\det F_e(\bar{\bm{\Xi}}) > 0$ for all $e$, which means $\widetilde{\bm{\Xi}} \in \mathcal{S}_{ad}^h$.


\end{proof}

Under the same condition, we can prove the convergence of series $\{ \bm{\Xi}^n \}$ for the discrete scheme for the given triangulation and fixed $\tau$.
\begin{proposition}
  For the given triangulation and fixed $\tau$, if the free energy density $\mathcal{W}(\X, F)$ (see (\ref{Energy_F}) for the definition) satisfies $\mathcal{W}(\X, F)  > 0$ for $\det F > 0$ and 
  \begin{equation}
    \mathcal{W}(\X, F) \rightarrow 0, \quad \det F \rightarrow \infty,
  \end{equation}
  for given $\varphi_0(\X)$, then the series $\{\bm{\Xi}^n\}$ converges to a stationary solution of the discrete energy $\mathcal{F}_h(\bm{\Xi})$.
\end{proposition}

\begin{proof}
  We first prove that there exist $c_0$ such that
  \begin{equation*}
      \bm{\Xi}^{\rm T} {\sf D}_n^* \bm{\Xi} \geq c_0 \| \bm{\Xi} \|^2, \quad \forall \bm{\Xi} \in \mathbb{R}^K.
  \end{equation*}
Since ${\sf M}_n^*$ is non-negative (see Appendix. A for the proof), we only need to show, for $( \sf K_0)^*_n$, there exist $c'_0$ such that
\begin{equation*}
  \bm{\alpha}^{\rm T} ({\sf K_0})_n^* \bm{\alpha} \geq c'_0 \| \bm{\alpha} \|^2, \quad \forall \bm{\alpha} \in \mathbb{R}^N. 
\end{equation*}
Indeed, note $\mathcal{F}_h(\bm{\Xi}^n) \leq \mathcal{F}_h (\bm{\Xi}^0)$, we have
   \begin{equation*}
    \begin{aligned}
    & \mathcal{F}_h (\bm{\Xi}^0) \geq \mathcal{F}_h(\bm{\Xi}^{n}) \geq   \mathcal{W}(\X_e, F_e) |\tau_e|, \\
    \end{aligned}
   \end{equation*}
following the same argument in the proof of the previous theorem,  we can show that $\det F_e$ is uniformly bounded away from zero, i.e., there exists $c_b > 0$ such that $\det F_e(\bm{\Xi}^{n}) > c_b$ for $\forall e$, . 
Hence,
\begin{equation*}
  \bm{\alpha}^{\rm T} ({\sf K_0})_n^* \bm{\alpha} = \sum_{e = 1}^M \int_{\tau_e} \Big|\sum_{i=1}^N \alpha_i \nabla \psi_i \Big|^2 \det F_e \dd \X \geq c_b \int_{\Omega} \Big|\alpha_i \nabla \psi_i \Big|^2 \dd \X \geq c_b \lambda_1 \| \bm{\alpha} \|^2,
\end{equation*}
where $\lambda_1 > 0$ is the smallest eigenvalue of the stiff matrix.

Then by theorem \ref{Thm1}, we have
  \begin{equation}
    \begin{aligned}
      c_0 \| \bm{\Xi}^{n+1} - \bm{\Xi}^{n} \|^2 & <  \left( ({\sf D}_{*}^{n} (\bm{\Xi}^{n+1} - \bm{\Xi}^n),  \bm{\Xi}^{n+1} - \bm{\Xi}^n \right) \\
      & \leq 2 \tau (\mathcal{F}_h(\bm{\Xi}^n) - \mathcal{F}_h(\bm{\Xi}^{n+1})) \\
    \end{aligned}
  \end{equation}
  Hence,
  \begin{equation}
    \sum_{k=0}^n c_0 \| \bm{\Xi}^{k+1} - \bm{\Xi}^{k} \|^2 \leq 2 \tau (\mathcal{F}_h(\bm{\Xi}^0) - \mathcal{F}_h(\bm{\Xi}^{n+1}) ) \leq C,
  \end{equation}
  where $C$ is independent with $n$. So
  \begin{equation}
    \lim_{k \rightarrow \infty}  \| \bm{\Xi}^{k+1} - \bm{\Xi}^{k} \|^2  = 0,
  \end{equation}
  which indicates the series $\{ \bm{\Xi}_k  \}_{k=1}^{\infty}$ converges to some point in $\mathbb{R}^K$, denoted by $\bm{\Xi}^{*}$. Following the same argument in the proof of the theorem \ref{Thm1}, we can show $\bm{\Xi}^{*} \in \mathcal{S}_{ad}^h$. Moreover, since
  \begin{equation}
  \lim_{n \rightarrow \infty} \frac{\delta \mathcal{F}_h}{\delta \bm{\Xi}} (\bm{\Xi}^{n+1})  =  \lim_{n \rightarrow \infty} - \frac{1}{\tau} D_n^* (\bm{\Xi}^{n+1} - \bm{\Xi}^{n}),
  \end{equation}
  we have $\frac{\delta \mathcal{F}_h}{\delta \bm{\Xi}} (\bm{\Xi}^*) = 0$, so $\bm{\Xi}^*$ is a stationary solution of the discrete energy $\mathcal{F}_h(\bm{\Xi})$.
\end{proof}

  It should be emphasised that the condition (\ref{Assumption_WF}) doesn't hold for classical phase-field free energy. Hence, it might be difficult to show the existence of the numerical scheme that minimizer the $J_n({\bm \Xi})$. Moreover, even the minimizer of $J_n({\bm \Xi})$ 
  exists, our line-search based optimization cannot guarantee to find it in each iteration. 
This is a limitation of the current numerical approach.
In practice, we choose a small value of $\tau$ and large value of $\nu$ such that the optimization problem can be handled by a standard optimization method, such as L-BFSG. Indeed, the first term in (\ref{define_Jn}) can be viewed as a regularization term, which restricts us to find a minimizer around $\bm{\Xi}^n$. The positive-definite condition on ${\sf D}^*_n$ is crucial,  otherwise, $J_n(\bm{\Xi})$ may have infinite minimizer even around $\bm{\Xi}^n$. In all numerical experiments shown in the next section, we adopt L-BFGS with line search to find a minimizer $\bm{\Xi}^{n+1}$ in the admissible set that decreases the discrete energy. 
The Lagrangian calculations will terminate if no $\bm{\Xi}^{n+1}$ is found or $|\mathcal{F}_h({\bm \Xi}^{n+1}) - \mathcal{F}_h({\bm \Xi}^{n})| \leq \epsilon$, where $\epsilon$ is the given tolerance.


\subsection{Reinitialization}
In the numerical implementation, we can compute $\x_h^{n+1}$ by
\begin{equation}
\x_h^{n+1}(\X) = \widetilde{\x}^{n+1}_h \circ \x^{n}_h(\X),
\end{equation}
which is equivalent to set $\X_i = \xi^n_i$ after each iteration as in \cite{junge2017fully}. 
An advantage of this treatment that in each iteration, we only need to compute a close to identity map \cite{junge2017fully}. So the optimization problem (\ref{Min_Problem}) is often easy to solve.

One can view this as a reinitialization procedure. More complicated reinitialization procedure can be incorporated in our numerical framework. Indeed, for given $\x^{n}(\X)$, we also obtain the numerical solution $\phi^n$ defined at mesh points, that is
\begin{equation}
\phi^n(\bm{\xi}_i^n) = \phi_0(\X).
\end{equation}
When the mesh become too skew, we can interpolate the numerical solution $\phi^n$ into a more regular mesh, obtained by coarsening or refining the current mesh \cite{bank1996algorithm}. More importantly, we can also apply Eulerian solver by using $\phi^n$ as the initial condition, to update the value at each mesh point. This is close to the idea in velocity-based moving mesh method \cite{baines2005moving}, which update both positions and values of mesh points. Unlike the traditional velocity-based moving mesh methods, our solution is spontaneously updated when the mesh moves.
We'll explore reinitialization procedures in details in the future work.

\section{Numerical validation and discussion}\label{sec:results}
In this section, we apply our Lagrangian scheme to several problems modeled by Allen-Cahn type phase-field models. Most of numerical examples used here are widely studied by Eulerian methods previously \cite{chen1998applications,feng2006spectral, shen2009efficient, di2008general, zhang2007adaptive}.
Numerical results show that, by choosing a suitable initial condition, our methods can capture the thin diffuse interfaces with a small number of mesh points, and reach a desired equilibrium.

Since we might apply a few Eulerian step in following numerical examples. here we brief introduce the Eulerian method that we'll use.
There are a lot of Eulerian methods for Allen-Cahn type phase-field model. By the spirit of ``discrete-then-variation'' approach, here we use an Eulerian solver derived by the discrete energetic variational approach. We use the same finite element space with the Lagrangian solver, and approximate the phase variable $\varphi$ by
\begin{equation}\label{approximate_phi}
  \varphi_h(X, t) = \sum_{i=1}^N \gamma_i(t) \psi_i(X),
\end{equation}
where $\psi_i(X)$ are hat functions on the current mesh.
Inserting (\ref{approximate_phi}) into the continuous energy-dissipation law, we can get an discrete energy-dissipation law with the discrete energy and the discrete dissipation given by
\begin{equation}
  \begin{aligned}
  & \mathcal{F}_h^{\rm{Euler}} = \sum_{e = 1}^N \int_{\tau_e} \frac{1}{2} \Big| \sum_{i=1}^N \gamma_i \nabla \psi_i(\X) \Big|^2 + \frac{1}{\epsilon^2} \sum_{i=1}^N (\gamma_i^2 - 1)^2 \psi_i(\X)  \dd \X, \\
  & \mathcal{D}_h^{\rm{Euler}} = \sum_{e = 1}^N \int_{\tau_e} \Big| \sum_{i=1}^N \gamma_i'(t) \psi_i(\X) \Big|^2 \dd \X, \\
  \end{aligned}
\end{equation}
respectively, where we also introduce the piecewise linear approximation to the nonlinear term in the discrete energy. This form of discrete energy was used in \cite{xu2019stability} and has an advantage in preserving the maximum principle at the discrete level \cite{xu2019stability}. After we obtain the semi-discrete equation of $\gamma_i(t)$, we solve it by implicit Euler method, which can also be reformulated into a minimization problem, similar to (\ref{Min_Problem}). Indeed, the Eulerian solver we used here is close to that in \cite{xu2019stability}.

\subsection{Quasi-1D example}
First, we consider a quasi-1D problem, in which $\Omega = [-1, 1]^2$. We impose Dirichlet boundary condition on $x = -1$ and $1$, that is
\begin{equation}\label{BD1}
\varphi(t, -1, y) = -1, \quad \varphi(t, 1, y) = 1,
\end{equation}
and Neumann boundary condition on $y = -1$ and $1$, that is
\begin{equation}\label{BD2}
\frac{\pp \varphi}{\pp y}(t, x, \pm 1) = 0.
\end{equation}
If the initial condition $\varphi_0(\X)$ satisfies (\ref{BD1}) and (\ref{BD2}), we can impose the boundary condition
\begin{equation}\label{BD_xy}
\x(\pm 1, Y, t) = (\pm 1, Y), \quad  \x(X, \pm 1) = (x, \pm 1)
\end{equation}
for the flow map $\x(\X, t): (X, Y)  \mapsto (x, y)$ such that $\varphi(\x(\X, t))$ satisfies (\ref{BD1}) and (\ref{BD2}). The boundary condition (\ref{BD_xy}) can be satisfied if $\uvec(X, Y)$ satisfies
 \begin{equation}
    \uvec(t, \pm 1, Y) = 0, \quad \uvec(t, X, \pm 1) \cdot \bm{n} = 0,
\end{equation}
where $\bm{n} = (0, 1)^{\rm{T}}$. In the following, we take the initial condition as
\begin{equation}
\varphi_0(X, Y) = - \tanh (5 X);
\end{equation}

\begin{figure}[!h]
  \includegraphics[width = \linewidth]{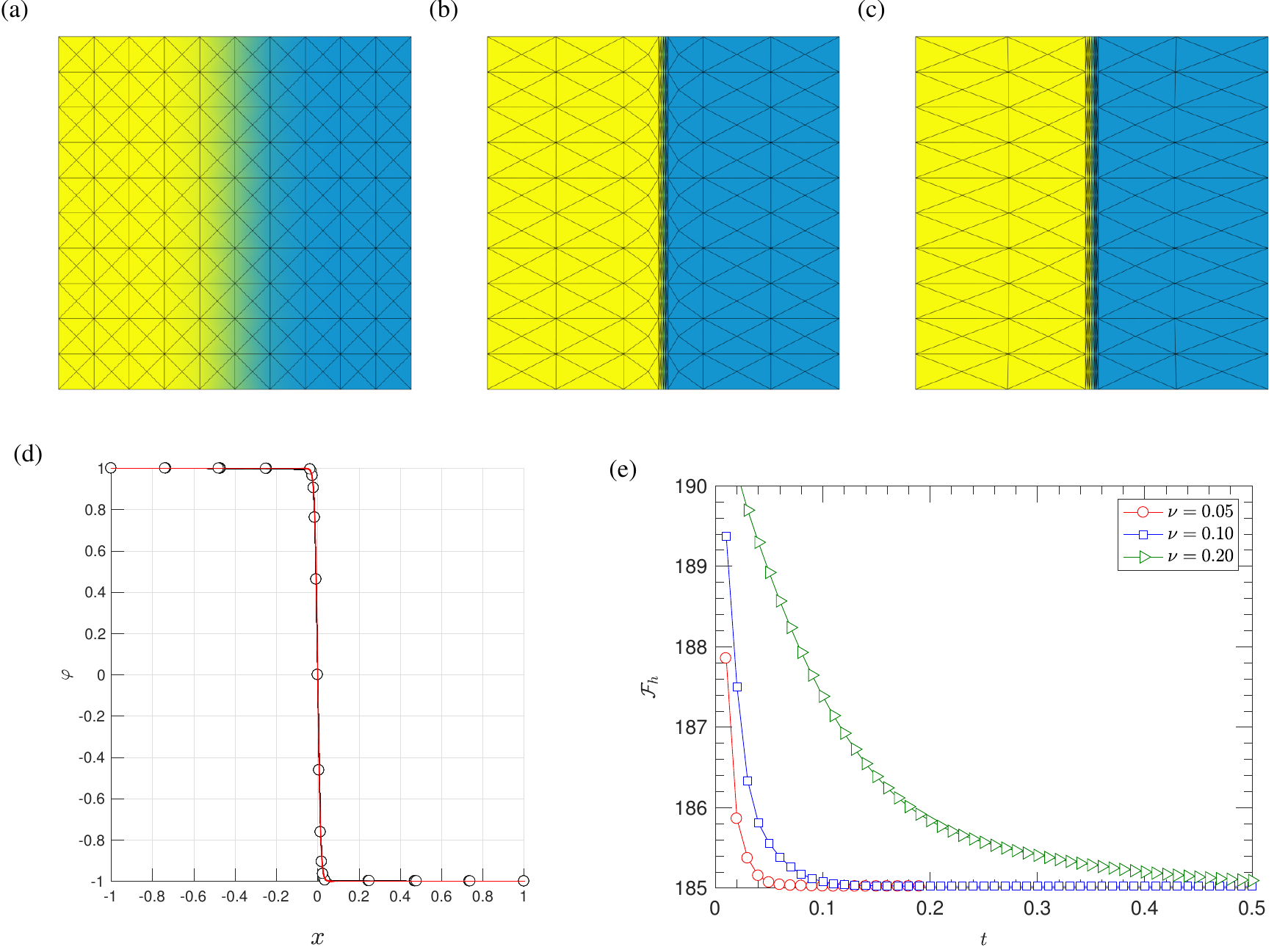}
  \caption{Numerical results for the quasi-1D example with $\epsilon^2 = 1e-4$. (a) - (c) Initial conditions and numerical solution at $t = 0.01$ and $t = 0.2$ (equilibrium). (d)Side view of the equilibrium solution, compared with the 1D exact solution. (e) Discrete free energy $\mathcal{F}_h$ along with time for various $\nu$.  }\label{Res1}
\end{figure}

Typical meshes and computed solution for $\epsilon^2 = 10^{-4}$ and $\nu = 0.05$ at $t = 0, 0.01$ and $0.2$ are shown Fig. \ref{Res1} (a)-(c). 
The initial mesh is the uniform mesh with $M = 400$.
We compare the obtained equilibrium solution with the 1D exact solution $\varphi(x) = - \tanh\left( \frac{x}{\sqrt{2} \epsilon} \right)$ in Fig. \ref{Res1} (d), in which the circles represent the projection of mesh points in the x-z plane, and the red line is the exact solution. It can be noticed that the equilibrium numerical solution can capture the thin interface with a small number of mesh points. 
Due to the presence of $|\nabla \uvec|^2$ term in the dissipation, the vertical velocity of all mesh points are almost zero, which is essential for a successful Lagrangian computation in this case. 

Fig. \ref{Res1} (e) shows the discrete energy as a function of time $t$ for different values of $\nu$. One can notice that our scheme is energy stable in all cases and all calculations go to the same equilibrium. The convergence to the equilibrium becomes slower when $\nu$ become larger. On the other hand, numerical tests show that the optimization problem (\ref{Min_Problem}) in each iteration will be easier to solve for larger $\nu$. In general, the value of $\nu$ also effect the equality of the obtained mesh.
We are not going to discuss the choice of $\nu$ in this paper, in the following, we choose larger $\nu$ for smaller $\epsilon^2$.

 Compared with Eulerian method, the Lagrangian methods has advantage in capture the diffuse interface with a small number of number points. However, the numerical approximation in the bulk region might be poor since most of the mesh points are concentrated in the interface region. To illustrate this, we perform a accuracy test for this example. Since the solution is y-invariant, we take $\Omega = [-1, 1] \times [-0.1, 0.1]$ in the accuracy test, and only look at the numerical error for the equilibrium solution on $y = -0.1$ in the following. In Fig. \ref{Error1}(a) - (b), we show show the numerical error obtained by the Lagrangian (blue circles) at each mesh point on $y = - 0.1$ with $h = 0.2$ and $h = 0.1$ respectively. It can be noticed that in both cases, the numerical error attain its maximum 
 at the transition area between the diffuse interface and  the bulk region. We can apply the Eulerian solver to the obtained solution, the numerical error after applying the Eulerian solver is shown in Fig. \ref{Error1} by red squares. Interestingly, although the Eulerian solver decrease the $L^2$-error of the numerical solution (so is the discrete free energy), the numerical error near the interface might increase a bit. 
 This simple numerical test suggests that the Lagrangian methods has advantage in capturing the thin diffuse interface, while the Eulerian methods can achieved better numerical approximation to the solution in the bulk region. 
 We should emphasize that for the phase-field type model, Eulerian methods cannot obtain a right results if the mesh size is larger than the diffuse interface. So it is a nature idea to combine the Lagrangian method with some Eulerian method.

 \begin{figure}[!h]
  \centering
  \begin{overpic}[width = 0.9 \linewidth]{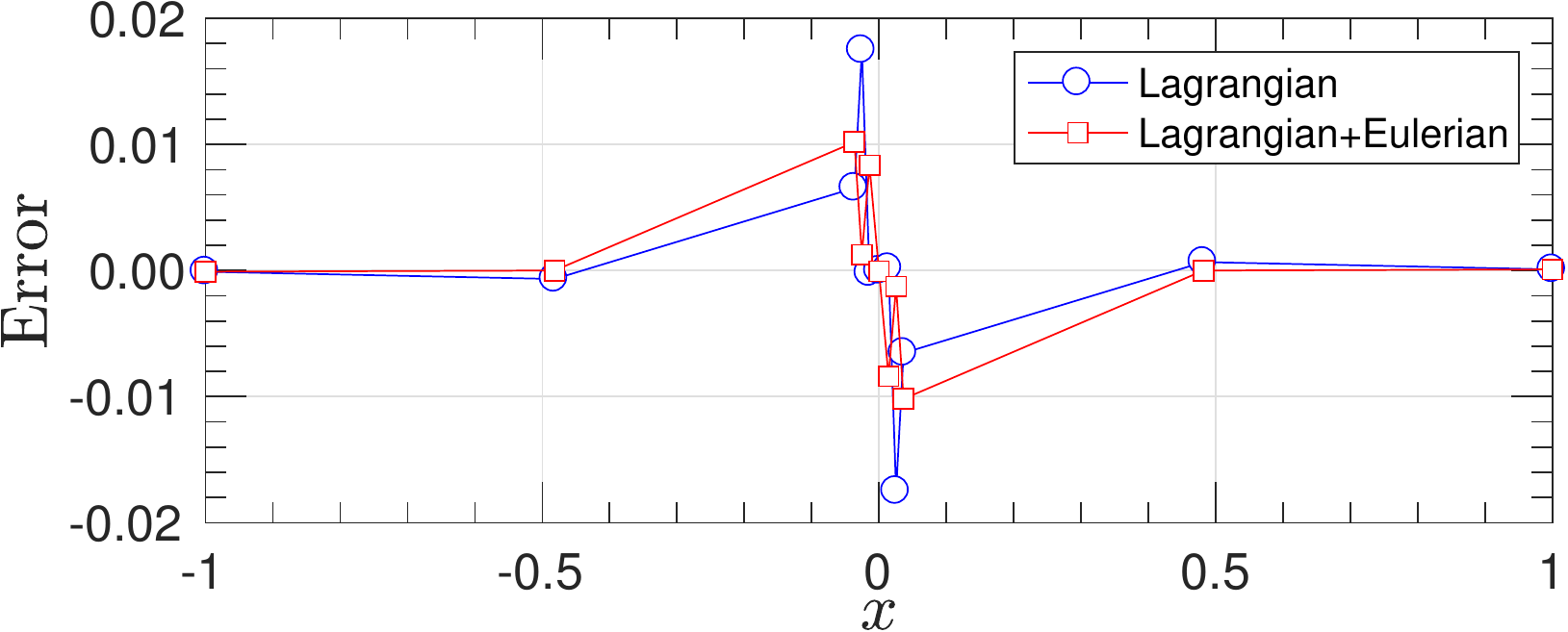}
    \put(-2, 40){\large (a)}
  \end{overpic}

  \vspace{1.5em}
  \begin{overpic}[width = 0.9 \linewidth]{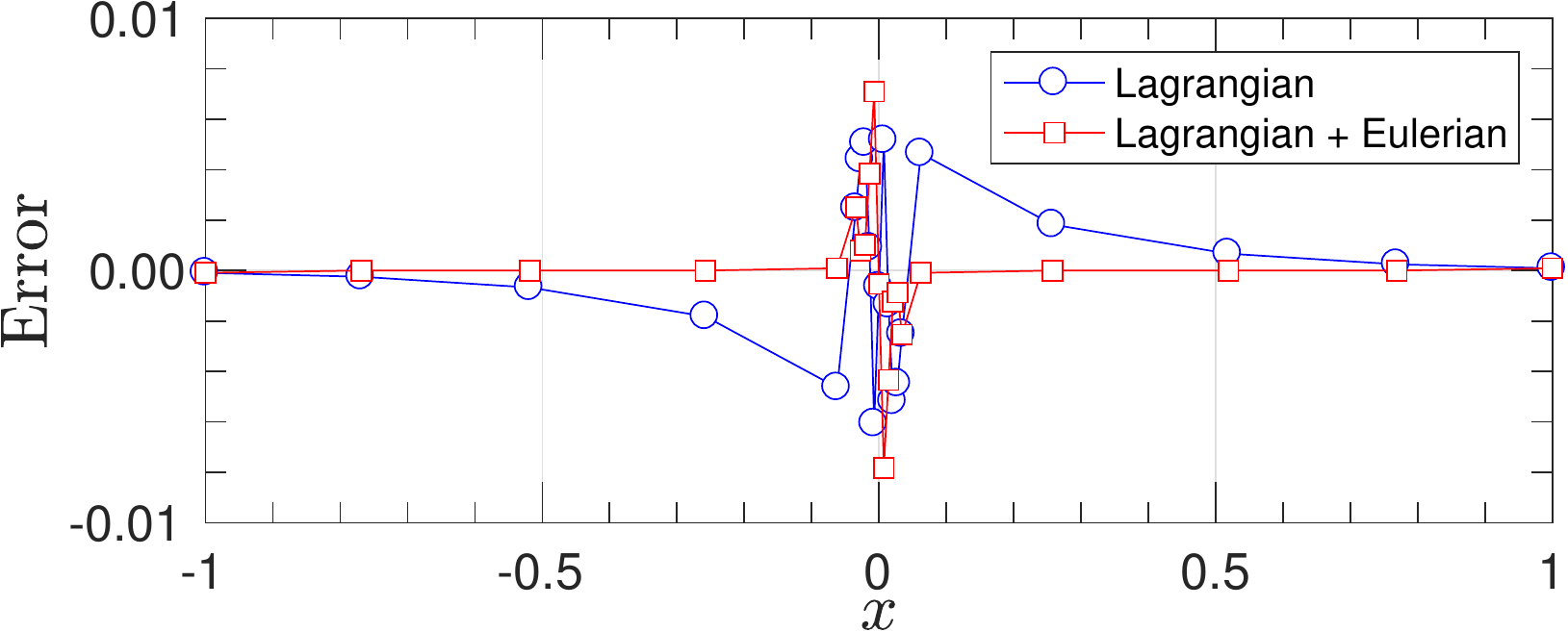}
    \put(-2, 40){\large (b)}
  \end{overpic}
  \caption{Numerical Error on $y = -0.1$ for the equilibrium solution, obtained by Lagrangian and Lagrangian+Eulerian methods: (a) $h = 0.2$ and (b) $h = 0.1$.   }\label{Error1}
\end{figure}

 We quantify the numerical error for Lagrangian method with different choice of $h$ for $\epsilon^2 = 10^{-3}$ and $10^{-4}$ . 
 The error in space is measured by the $L^{\infty}-$ norm defined by
 \begin{equation*}
 \|e_h\|_{\infty} = \max_{i \in \{ 1, 2, \ldots N \}} |\varphi^h(\x_h(\X_i, T)) - \varphi^{eq} (\x_h(\X_i))|,
 \end{equation*}
 where $T$ is the final time for the Lagrangian calculations.  
 Here we only test the convergence rate near the interface. It can be noticed that near the interface, our Lagrangian method can achieve second order in space. Another interesting phenomenon is that the numerical perform seems to be independent with $\epsilon$. We should also mention that the numerical error also is sensitive to the choice of the initial condition $\varphi_0(\X)$. A detailed numerical analysis is needed in order to understand these phenomena for the Lagrangian method.
  

\begin{table}[!h]
  \begin{center}
  \begin{tabular}{c  c | c  c | c  c |  c  c}
    \hline
    \multicolumn{4}{c|}{$\epsilon = 10^{-3}$}    &   \multicolumn{4}{c}{$\epsilon = 10^{-4}$}  \\ \hline
    h      &  $\tau$   &  $L^{\infty}$-error  & Order           &  h       &  $\tau$      & $L^{\infty}$-error & Order     \\ \hline
    0.2    &  1/100  &  0.0185            &        &  0.2     &  1/100   & 0.0175               &             \\  \hline  
    0.1    &  1/400   &  0.0059         &  1.6487      &  0.1     &  1/400    &  0.0052              &    1.7508            \\ \hline
    0.05   &  1/1600   & 0.0015       &  1.9758        &  0.05    &  1/1600   &  0..0015          &   1.7935     \\  \hline        
  \end{tabular}
  \caption{The convergence rate of numerical solutions near the interface ( $x \in [-3\epsilon, 3 \epsilon]$) with $\nu = 0.05$.}
  \label{Table3}
\end{center}
\end{table}




\subsection{Shrinkage of a circular domain}

As a numerical test, we consider shrinkage of a circular domain in two-dimension. It is a classical benchmark problem for the Allen-Cahn equation \cite{chen1998applications,feng2006spectral, shen2009efficient}, in which the circular interface governed by the Allen-Cahn equation will shrink and eventually disappear.

\begin{figure}[!h]
  \includegraphics[width = \linewidth]{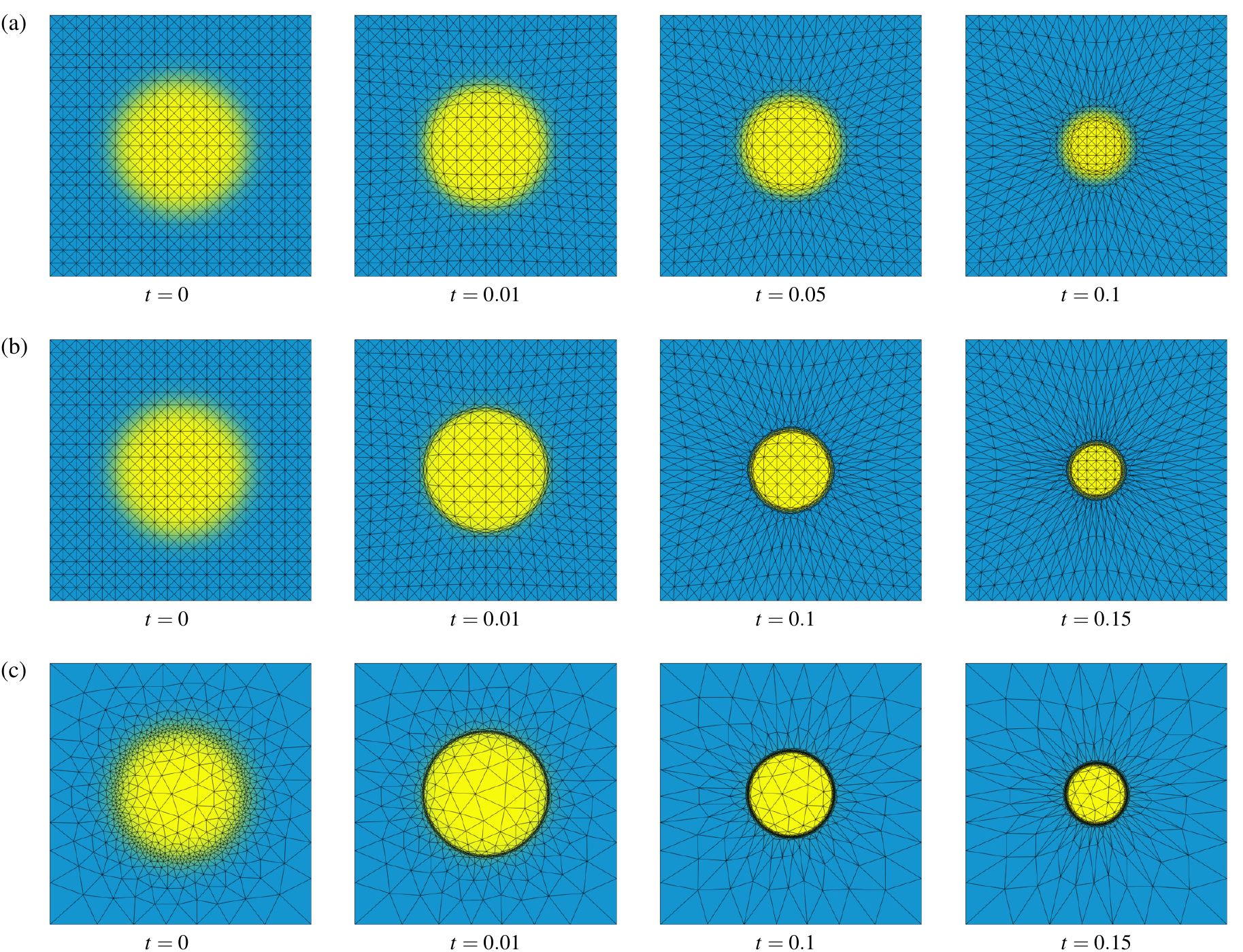}
  \caption{(a)  Numerical results of the evolution of the circular domain for $\epsilon^2 = 1e-3$ [$\tau = 10^{-2}$]. (b) - (c) Numerical results of the evolution of the circular domain for $\epsilon^2 = 10^{-4}$ [$\tau = 10^{-2}$]: (b) Uniform mesh, (c) Non-uniform mesh.}\label{Shirkage}
\end{figure}

We take $\Omega = [-1, 1]^2$ and impose the Dirichlet condition $\varphi(\x) = -1, \x \in \pp \Omega$. The initial condition $\varphi_0(\X)$ is taken as
\begin{equation}
\varphi_0(\X) = \tanh (10 (\sqrt{X^2 + Y^2} - 0.5) ),
\end{equation}
such that the Dirichlet condition satisfies numerically. 
It is worth pointing out that in our Lagrangian methods, it is crucial to choose a proper initial condition. For the phase model, it is often choose $\varphi_0(\X)$ in a hyperbolic tangent form such that $\varphi_0 \in [-1, 1]$, and the width of initial interface should be larger than the mesh size, since we need enough mesh points in the region of interface.

  Fig. \ref{Shirkage} (a) shows the numerical results for $\epsilon^2 = 10^{-3}$ with $\nu = 1$ at various time in a uniform mesh ($M = 1600$), while Fig. \ref{Shirkage} (b) and (c) show the numerical results for $\epsilon^2 = 10^{-4}$ with $\nu = 10$ in uniform ($M = 1600$) and non-uniform meshes ($M = 1348$) respectively. The non-uniform mesh is generated by DistMesh \cite{persson2004simple}. We choose larger $\nu$ for smaller $\epsilon^2$ to control the quality of the mesh. It can be noticed that in all three cases, the mesh points will be concentrated at thin interface after one time iteration and maintain concentrated at the moving interface all the time. The results in Fig. \ref{Shirkage} (c) suggest that we can incorporate our Lagrangian method with adaptive mesh technique. Within the Lagrangian solver,  we only need to adapt the initial mesh.
As a limitation, for this problem the Lagrangian calculation cannot reach the equilibrium, in which the circular domain is disappeared. Such a problem can be handled easily by applying some Eulerian solver to the numerical solutions obtained by Lagrangian calculations at the late stage.

\begin{figure}[!h]
  \centering
  \includegraphics[width = 0.75 \linewidth]{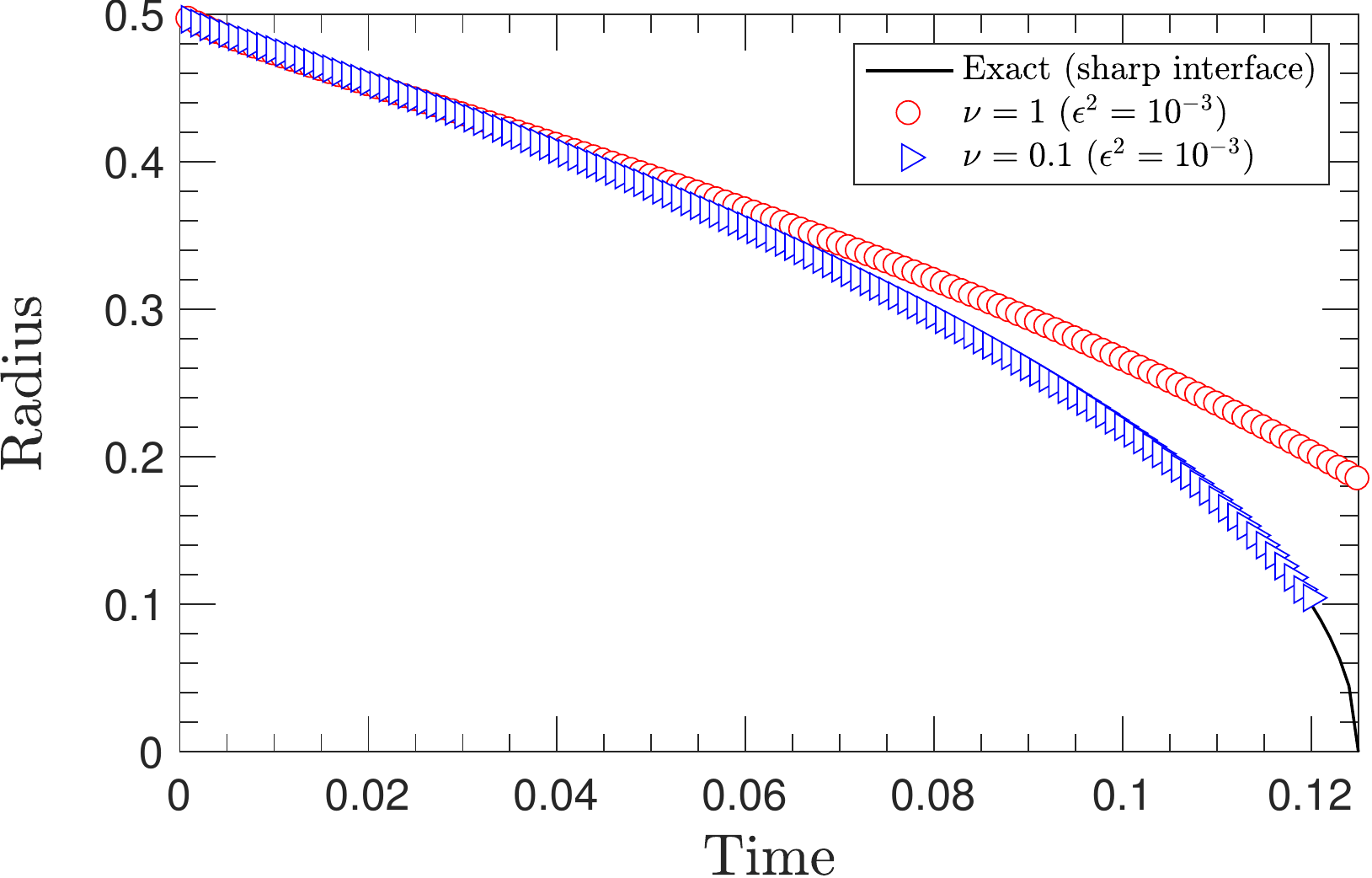}
  \caption{The evolution of interface as a function of time by our Lagrangian method for $\epsilon^2 = 10^{-3}$ and the different choice of $\nu$. The black line show the evolution of interface driven by mean curvature. } \label{Shirkage_time}
\end{figure}

Although we are mainly interested in the equilibrium solutions obtained by Lagrangian methods, we also compare the dynamics of Lagrangian approach with the original Eulerian approach by studying the evolution of diffuse interface in this examples.
It is well known that at the sharp interface limit, the movement of interface is driven by mean curvature flow, and $R(t) = \sqrt{R_0^2 - 2t}$, where $R(t)$ denotes the radius of the interface at time $t$ \cite{xu2019stability}. The singularity happens at $t = R_0/2$, which is the disappearing time. We compare the radius of the interface obtained by our numerical calculations for $\epsilon = 10^{-3}$ with $R(t)$ for different choices of $\nu$. We refer the authors to a similar comparison for some Eulerian methods. It can be noticed that for $nu = 0.1$, the evolution of sharp interface can be well approximated by our methods with small number of mesh points. Indeed, our initial mesh size is larger than $\epsilon$, it is impossible to get the right result by using Eulerian methods on this mesh \cite{merriman1994motion}. 
For $nu = 1$, it is expected that the movement of interface is slower, similar to previous example [see Fig. \ref{Res1} (c). Indeed, for large $\nu$, the second term in the dissipation actually dominate the dynamics of the Lagrangian method.



\subsection{Phase-field model with the volume constraint}
In this subsection, we consider an Allen-Cahn type phase-field model with the volume constraint. 
We impose the volume constraint by introducing a penalty term in the free energy. So the total free energy of the system is given by
   \begin{equation}\label{Energy_VolC}
        \begin{aligned}
          \mathcal{F} =   \int_{\Omega} \frac{1}{2} |\nabla \phi|^2 + \frac{1}{4 \epsilon^2} (\phi^2 - 1)^2 \dd \x + W_b \left(\int \phi \dd \x - A \right)^2. \\
        \end{aligned}
   \end{equation}
   We take $\Omega = [-1, 1]^2$, $W_b = 1000$, $A = -3$ and $\epsilon^2 = 10^{-4}$,  and impose the Dirichlet boundary condition $\varphi(\x) = -1, \x \in \pp \Omega$, throughout this section.

\begin{figure}[!ht]
  \includegraphics[width = \linewidth]{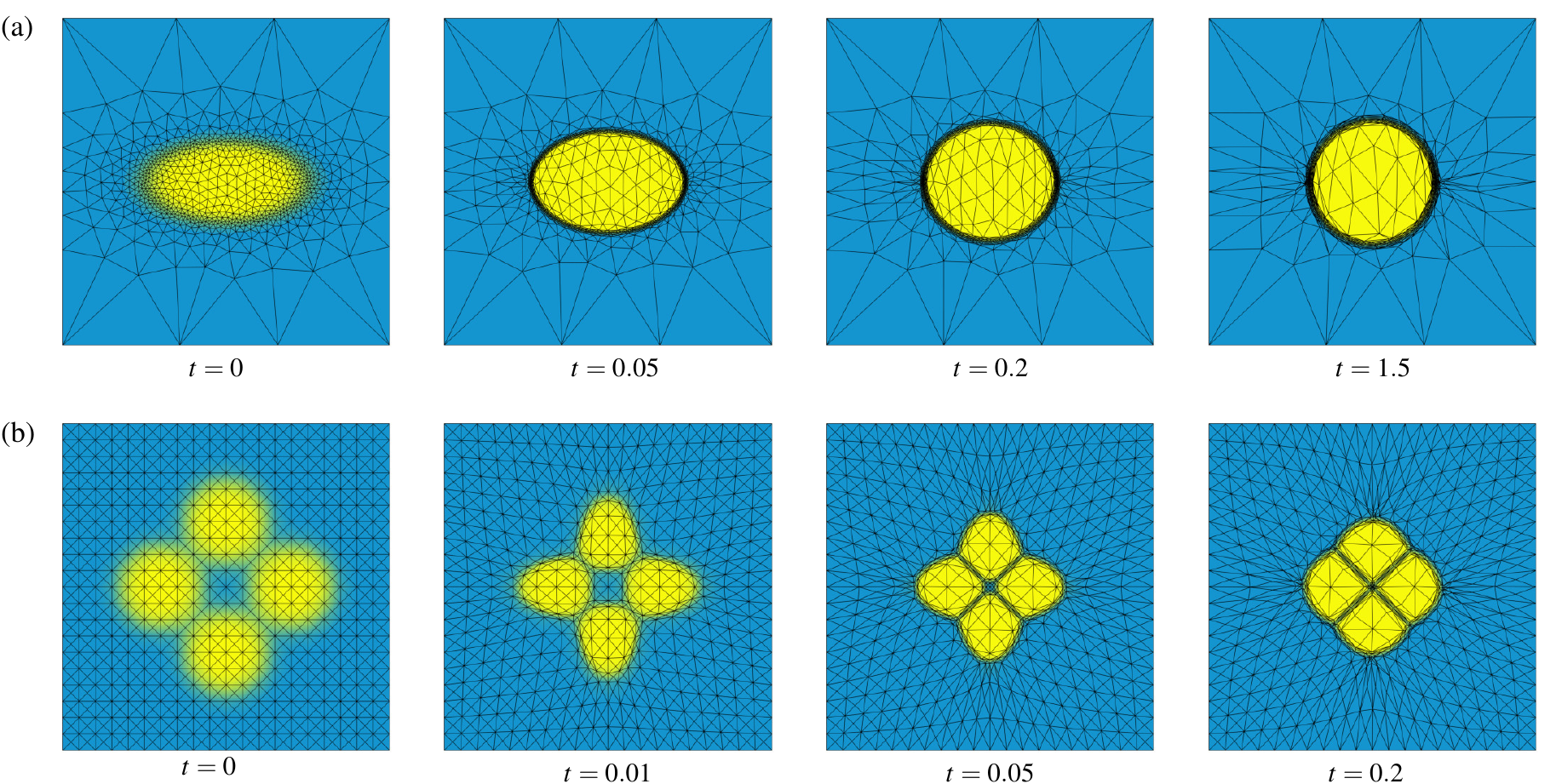} 
  \caption{Numerical results for phase-field model with volume constraints for $\epsilon^2 = 10^{-4}$ with different initial conditions [$\tau = 10^{-2}$]: (a) Single bubble. (b) Coalescence of four kissing bubbles.}\label{VolC}
\end{figure}

Fig. \ref{VolC} (a) shows numerical results for initial condition
   \begin{equation}
     \varphi_0(X, Y) = - \tanh( 10 (\sqrt{X^2 + 4 Y^2} - 1/2) ),
   \end{equation}
  in which we use a non-uniform mesh ($M = 1484$) generated by the DistMesh \cite{persson2004simple}
As expected, due to the effect of surface tension and the volume constraints, the bubble deforms into a circular bubble, and the mesh points keep concentrated at the thin interface when the shape of interface changes.
As a benefit of pure Lagrangian calculation, we can guarantee the numerical solution $\varphi_h(\x, t) \in [-1, 1]$.

We also consider the initial condition
\begin{equation}
\varphi_0(X, Y) = - \sum_{i=1}^4 \tanh(15 (r_i - 1/3)) + 3,
\end{equation}
where $r_1 = \sqrt{(X - 0.4)^2 + Y^2}$, $r_2 = \sqrt{(X + 0.4)^2 + Y^2}$, $r_3 = \sqrt{X^2 + (Y - 0.4)^2}$ and $r_4 = \sqrt{X^2 + (Y + 0.4)^2}$. This is also a classical test problem in phase-field models \cite{liu2003phase, zhang2007adaptive}, which correspond to coalescence of four kissing bubbles. As time evolves, the four bubbles are expected to coalesce into a big bubble. The initial condition and numerical solutions obtained by pure Lagrangian calculations at various $t$ are shown in Fig. \ref{VolC} (b) [Uniform mesh, $M = 1600$]. Although the mesh points are concentrated immediately at the thin interface, the natural of our Lagrangian methods prevent four bubbles merging together. This is a limitation of our Lagrangian methods which can not handle topological changes in the phase-field model, since the kinematic (\ref{Transport}) is only valid locally. 


\begin{figure}[!th]
  \includegraphics[width = \linewidth]{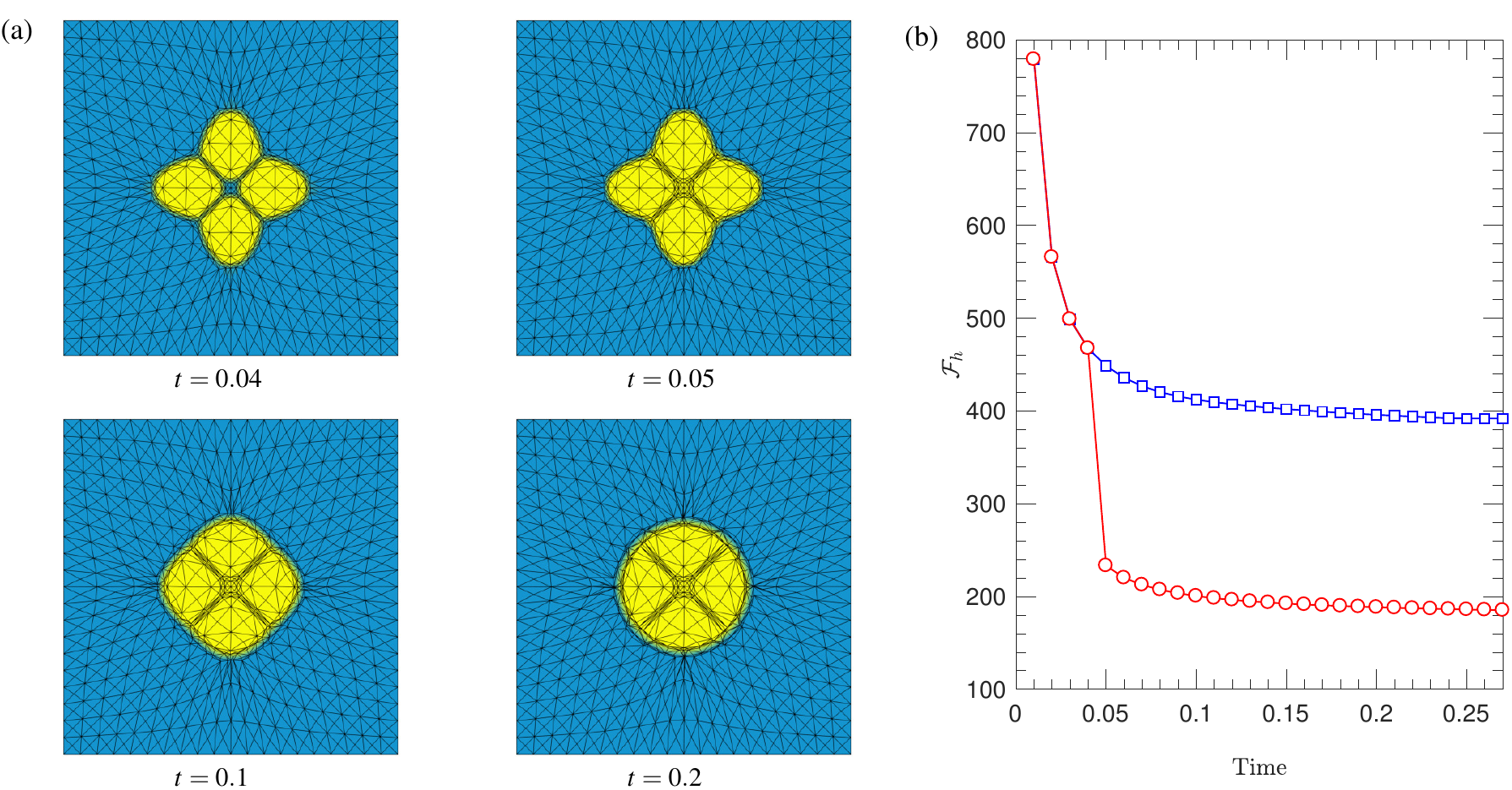}
  \caption{(a) Numerical results of ``coalescence of two kissing bubbles'' with an Eulerian step [$\epsilon^2 = 10^{-4}$, $\tau = 10^{-2}$]. (b) The total discrete energy with and without the Eulerian step.}\label{ALE}
\end{figure}

This drawback can be easily overcome by running a few Eulerian steps on the obtained mesh at the reinitialization step. 

A subtle issue is when shall we apply the Eulerian solver, which is problem-dependent in general. For phase-field models, the Eulerian steps are necessary to handle the topological changes, like the examples in Fig. \ref{Shirkage} and Fig. \ref{VolC} (b). 
In the current study, we are not going to discuss this issues in details. For the test problem shown in Fig.\ref{VolC} (b), we actually only need one Eulerian step to handle the topological change.
Fig. \ref{ALE} (a) shows the numerical results with  applying the Eulerian step at the fifth step ($t = 0.05$). Since we only do one Eulerian step, we didn't include the penalty term in (\ref{Energy_VolC}) to ensure maximum principle is preserved at the discrete level. Fig. \ref{ALE} (b)  shows the computed total discrete energy with and without the Eulerian step. It can be noticed that Eulerian step significantly decreases the discrete energy. A better performance can be achieved by applying local mesh coarsening \cite{bank1996algorithm}. 

It is worth mentioning that for some particular problems modeled by phase-field methods, such as colloidal particles in liquid crystals \cite{yue2004diffuse, zhou2008dynamic, wang2018formation}, in which the coalescence of colloidal particles is often not allowed, it might be an advantage to use our Lagrangian scheme to prevent the topological changes.

\subsection{``Slightly compressible'' flow}
In the final example, we consider a phase-field model with the total energy given by
\begin{equation}
 \int_{\Omega_0} \frac{1}{2} |F^{-\rm T} \nabla_{\X} \varphi_0|^2 + \frac{1}{4 \epsilon^2} (\varphi^2 - 1)^2 + \eta \left( \frac{1}{\det F} - 1 \right)^2 \det F \dd \X.
\end{equation}
where the last term can be viewed as a penalty term for the incompressible condition $\nabla \cdot \uvec$ in the incompressible two-phase flow \cite{liu2003phase, hyon2010energetic}. 
One can noticed that this form of free energy satisfies (\ref{{Assumption_WF}}).
This model is analogous to slightly compressible two phase flow \cite{temam2001navier}.

\begin{figure}[!hbt]
  \includegraphics[width = \linewidth]{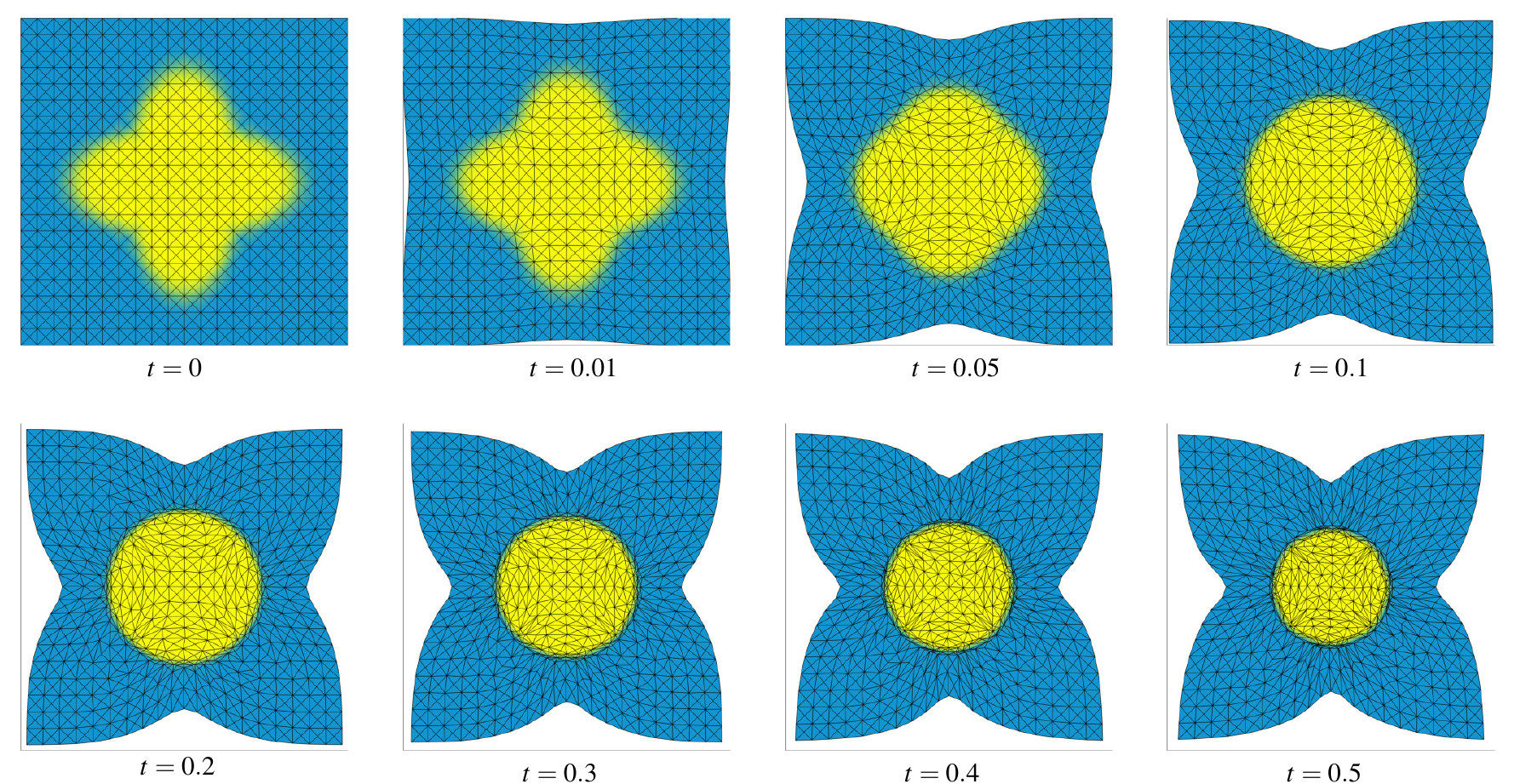} 
  \caption{Numerical results for ``slightly compressible'' phase-field model for $\epsilon^2 = 10^{-4}$ [$\tau = 10^{-2}$] at various time.}\label{InCompressible}
\end{figure}

Fig. \ref{InCompressible} shows numerical results for initial condition
 \begin{equation}
   \varphi_0(X, Y) = \max(- \tanh( 15 (r_1 - 0.7)), - \tanh( 15 (r_2 - 0.7))),
 \end{equation}
 where $r_1 =\sqrt{X^2 + 4 Y^2}$ and $r_2 = \sqrt{4 X^2 + Y^2}$, with $\epsilon^2 = 10^{-4}$ and $\nu = 10$. 
 Here, we impose the free boundary on the flow map $\x(\X, t)$, and take $\Omega_0 = [-1, 1]^2$, $\epsilon^2 = 10^{-4}$ and $\eta = 5000$. As expected, the bubble will also deform into a circular bubble, and shrink. Compared with previous examples, due to penalty terms on constraints of $\det F = 1$, the mesh will not immediately concentrate around the thin diffuse interface. One can also view the additional penalty term in the free energy as a regularization term, which improves the mesh quality.



\section{Summary}
In this paper, we propose a variational Lagrangian scheme to a phase-field model, which can compute the equilibrium states of the original Allen-Cahn type phase field model, with a proper choice of $\varphi(\X)$.
Numerical examples show that our scheme has advantage in capturing the thin diffuse interface in the equilibria with a small number of mesh points. 
Our approach can be extended to general gradient flow system, especially those have equilibria with singularity, sharp interface and free boundary, such as Landau-de Gennes model for liquid crystal \cite{macdonald2015efficient, wang2017topological, wang2019order} and Ginzburg–Landau model for superconductivity \cite{du2005numerical}. 

There are still many limitations of our methods. As mentioned previously, it is important to combine some Eulerian solver with the Lagrangian solver presented here. For Allen-Cahn like gradient flow system, choosing a proper initial data $\varphi_0$ is crucial in order to get a reasonable results, as the kinematic relations, i.e., the transport equations, in the Lagrangian approach may only valid locally. Moreover, a pure Lagrangian scheme are not able to deal with the large deformation and topological change. These drawbacks might be overcome by incorporating Eulerian steps into the Lagrangian calculations. Moreover, from a numerical approximation perspective, as shown in Fig. \ref{Error1}, Lagrangian methods have advantage in capture the interface, while Eulerian methods can a achieve better numerical approximations in the bulk region. So it is necessary to combine both approaches to achieve high accuracy with small computational efforts. The main difficulty in combining a Lagrangian solver with an Eulerian solver is to decide which one to apply during the computational procedure. Another drawback for our methods, the Lagrangian mesh may still become too skew even with the regularization $\nu |\nabla \uvec|^2$ term in the dissipation. The local mesh refinement or coarsening is needed to improve the robustness of the Lagrangian calculations. These points will be the subjects of future work.


\appendix 

\section{Derivation of (\ref{F_x})}
In this appendix, we provide a detailed calculation of (\ref{F_x}). For any smooth map $\y(\X, t) = \tilde{\y} (\x(\X, t), t)$, we denote
\begin{equation}
\x^{\epsilon} = \x + \epsilon \y, \quad F^{\epsilon} = \nabla_{\X} \x^{\epsilon}.
\end{equation}
Then we have
\begin{equation*}
\begin{aligned}
  & \frac{\dd}{\dd \epsilon} \Big|_{\epsilon = 0} \mathcal{F}(\x^{\epsilon})  =  \frac{\dd}{\dd \epsilon} \Big|_{\epsilon = 0} \left( \int_{\Omega^{\epsilon}} W(\varphi(\x^{\epsilon}), \nabla_{\x^{\epsilon}} \varphi(\x^{\epsilon})) \dd \x^{\epsilon} \right) \\
  & =  \frac{\dd}{\dd \epsilon} \Big|_{\epsilon = 0} \left( \int_{\Omega^{0}} W(\varphi_0(\X), (F^{\epsilon})^{-\rm{T}} \nabla_{X} \varphi_0)  \det F^{\epsilon} \dd \X \right) \\
  & =  \int_{\Omega^{0}}  - \left( \frac{\pp W}{\pp \nabla \varphi} \otimes  \nabla_{\X} \varphi \right) : F^{-\rm T} (\nabla_{\X} \y)^{\rm T}  F^{-\rm T}   \det F  +  W(\varphi,  \nabla \varphi) \det F F^{-\rm{T}} : \nabla_{\X} \y  \dd \X \\
  &  = \int_{\Omega^0} \det(F) \left( - \left( \frac{\pp W}{\pp \nabla \varphi} \otimes  F^{-\rm{T}} \nabla_X \varphi_0 \right) +  W(\varphi, \nabla \varphi) \rm{I} \right)   : (\nabla_{\X} \y F^{-1})^{\rm T}  \dd \X \\
  & = \int_{\Omega^0} \det(F) F^{-1}\left( - \left( \frac{\pp W}{\pp \nabla \varphi} \otimes  F^{-\rm{T}} \nabla_X \varphi_0 \right) +  W(\varphi, \nabla \varphi) \rm{I} \right)   : (\nabla_{\X} \y)^{\rm T}  \dd \X \\
\end{aligned}
\end{equation*}

Pushing forward the above result to the Eulerian coordinates and performing integration by parts, we have
\begin{equation}
\begin{aligned}
 \frac{\dd}{\dd \epsilon} \Big|_{\epsilon = 0} \mathcal{F}(\x^{\epsilon}) & = \int_{\Omega}\left( - \left( \frac{\pp W}{\pp \nabla \varphi} \otimes \nabla \varphi \right) +  W(\varphi, \nabla \varphi) \rm{I} \right) : \nabla \tilde{\y} \dd \x \\
  & = \int_{\Omega} \left( \nabla \cdot \left(  \frac{\pp W}{\pp \nabla \varphi} \otimes \nabla \varphi -  W(\varphi, \nabla \varphi) \mathrm{I} \right) \right)  \cdot \tilde{\y}  \dd \x, \\
  \end{aligned}
\end{equation}
where the boundary term vanishes due to the boundary condition of $\tilde{y}$ or $\varphi$. Hence,
\begin{equation}
\frac{\delta \mathcal{A}}{\delta \x} = - \frac{\mathcal{F}}{\delta \x} =  - \nabla \cdot \left( \left( \frac{\pp W}{\pp \nabla \varphi} \otimes \nabla \varphi \right)  - W \mathrm{I}  \right).
\end{equation}

\section{The form of semi-discrete equation} \label{App_B}
Here we provide detailed calculations to $\frac{\delta \mathcal{A}}{\delta \Xi_i}$ and $\frac{\delta \mathcal{\mathcal{D}}}{\delta \Xi_i^{\prime}}$ in each element $\tau_e$. The calculations are very close to the variation with respect to $\x$ and $\x_t$ in the continuous level.

Recall the discrete free energy $\mathcal{F}_h(\bm{\Xi}(t))$ and the discrete dissipation functional $\mathcal{D}_h(\bm{\Xi}(t), \bm{\Xi}^{\prime}(t))$ are given by
\begin{equation}\label{Action_1_h}
\begin{aligned}
  \mathcal{F}_h (\bm{\Xi}(t)) 
  & = \sum_{e = 1}^{M} \int_{\tau_e}  W(\varphi_0, F^{-\rm T}_{e} \nabla_{\X} \varphi_0) \det F_e \dd \X, \\
\end{aligned}
\end{equation}
and
\begin{equation}\label{Dis_1_h}
\begin{aligned}
  \mathcal{D}_h(\bm{\Xi}(t), \bm{\Xi}^{\prime}(t))  = \frac{1}{2} \sum_{e = 1}^M \int_{\tau_e}  & \Big| (F_e^{-\rm T} \nabla_{\X} \varphi_0) \cdot (\sum_{j = 1}^N \bm{\xi}_j' \psi_j(\X) ) \Big|^2  \\
  & + \nu \Big| \nabla_{\X} (\sum_{j = 1}^N \bm{\xi}_j' \psi_j(\X)) \Big|^2   \det F_e  \dd \X,
\end{aligned}
\end{equation}
respectively.
Let $N(i)$ be all the indices $e$ such that $\X_i$ is contained in $\tau_e$ for given $\X_i \in \mathcal{N}_h$. Then for $\chi_i = \xi_{i, x}$ or $\xi_{i, y}$, we have
\begin{equation*}
  \begin{aligned}
    \frac{\pp \mathcal{F}_h}{\pp \chi_i} & = \sum_{e \in N(i)} \int_{\tau_e} \frac{\pp}{\chi_i} \left(  W \left( \varphi_0(X), F_e^{-\rm{T}} \pp_{X} \varphi_0 \right) \det F_e \right) \dd X \\
     & =  \sum_{e \in N(i)} \int_{\tau_e}   F_e^{-1}\left(  \left( - \frac{\pp W}{\pp (\nabla_{\x} \varphi)} \otimes \nabla_{\X} \varphi_0 \right) F_e^{-1} + W   \mathrm{I}  \right): \left(\frac{\pp F_e}{\pp \chi_i} \right)^{\rm{T}} \det F_e \dd X. \\
  \end{aligned}
\end{equation*}
The numerical integration above can be computed by using centroid method.

Meanwhile, for the dissipation part, direct computation results in 
\begin{equation}
  \begin{aligned}
    \frac{\pp \mathcal{D}_h}{\pp \xi_{i, x}'} = \sum_{e \in N(i)} \int_{\tau_e} &  \sum_{j=1}^N \left( (\varphi_x^2 \psi_i(\X) \psi_j(\X)) \xi_{j, x}' + (\varphi_x \varphi_y \psi_i(\X) \psi_j(\X)) \xi_{j, y}'  \right)  \\
    & + \nu  (\nabla_{\X}  \psi_i \cdot \nabla_{\X} \psi_j) \xi_{i, x}' \det F_e \dd \X,  \\
  \end{aligned}
\end{equation}
and
\begin{equation}
  \begin{aligned}
    \frac{\pp \mathcal{D}_h}{\pp \xi_{i, y}'} = \sum_{e \in N(i)} \int_{\tau_e} &  \sum_{j=1}^N \left( (\varphi_x \varphi_y \psi_i(\X) \psi_j(\X)) \xi_{j, x}' + (\varphi_y^2 \psi_i(\X) \psi_j(\X)) \xi_{j, y}'  \right)  \\
    & + \nu  (\nabla_{\X} \psi_i \cdot \nabla_{\X} \psi_j) \xi_{j, y}' \det F_e \dd \X.  \\ 
  \end{aligned}
\end{equation}
Hence,
\begin{equation}
{\sf D}(\bm{\Xi}(t)) = {\sf M}(\bm{\Xi}(t)) + \nu {\sf K}(\bm{\Xi(t)})
\end{equation}
with
\begin{equation*}
  {\sf M} =
  \begin{pmatrix}
    {\sf M}_{xx} &  {\sf M}_{xy} \\
    {\sf M}_{yx} &  {\sf M}_{yy}, \\
  \end{pmatrix}
  \quad
    {\sf K} =
  \begin{pmatrix}
    {\sf K_0} &  {\sf 0} \\
     {\sf 0} &  {\sf K_0} \\
  \end{pmatrix}.
\end{equation*}
Here ${\sf M}_{\alpha \beta}$ ($\alpha, \beta = x, y$) is the modified mass matrix defined by
\begin{equation}
{\sf M_{\alpha \beta}}(i, j) = \sum_{e \in N(i)} \pp_{\alpha}\varphi(\x_e) \pp_{\beta}\varphi(\x_e)  \det F_e \int_{\tau_e} \psi_i \psi_j     \dd \X,
\end{equation}
where $\x_e$ is the centroid of $\x^t(\tau_e)$, and  ${\sf K_0}$ is the modified stiff matrix defined by
\begin{equation}
{\sf K_0}(i, j) = \sum_{e \in N(i)} \det F_e \int_{\tau_e} \nabla_{\X} \psi_i \cdot  \nabla_{\X} \psi_j   \dd \X.
\end{equation}

Next we show that positive definiteness of ${\sf D}$ if ${\bm \Xi} \in \mathcal{S}_{ad}^h$ and $\nu > 0$. We first show that ${\sf M}$ is positive semi-definite. Recall that ${\sf M}$ is obtained by the summing the results on each element the over mesh, we only need to show ${\sf M}^{e}({\bm \Xi})$ is positive semi-definite, for each $e \in \{ 1, 2, \ldots M \}$. The positive semi-definiteness of ${\sf M}^{e}({\bm \Xi})$ can be proved by looking at the principal minor formed by all non-zero element, which is $6 \times 6$ matrix given by
\begin{equation}
\widetilde{\sf M}^e = 
\begin{pmatrix}
  {\sf M}_{xx}^e & {\sf M}_{xy}^e \\
  {\sf M}_{xy}^e & {\sf M}_{yy}^e \\
\end{pmatrix},
\end{equation}
where $M_{\alpha \beta}^e$ is defined by
\begin{equation}
M_{\alpha \beta}^e = \frac{1}{12}  |\tau_e| \det F_e \left(  \pp_{\alpha} \varphi (\x_e) \pp_{\beta} \varphi (\x_e) \right) \begin{pmatrix}
  2 & 1 & 1 \\
  1 & 2 & 1  \\
  1 & 1 & 2 \\
\end{pmatrix} =   \pp_{\alpha} \varphi (\x_e) \pp_{\beta} \varphi (\x_e) {\sf M}^e_0,
\end{equation}
where ${\sf M}_0^e$ is positive definite matrix since $\det F_e > 0$. Due to the positive definiteness of ${\sf M}_0^e$, we can write $M_0^e$ as $M_0^e = \sqrt{{\sf M}_0^e} \sqrt{{\sf M}_0^e}$. 
Then $\forall {\bf a} = ({\bf a}_1^{\rm T}, {\bf a}_2^{\rm T})^{\rm T} \in \mathbb{R}^6$ with ${\bf a}_i \in \mathbb{R}^3$, by direct calculation, we have
\begin{equation*}
  \begin{aligned}
{\bf a}^{\rm T} \widetilde{\sf M}^e {\bf a} & = \left( \pp_x \varphi (\x_e) {\bf a}_1^{\rm T} \sqrt{{\sf M}_0^e} +  \pp_y \varphi (\x_e) {\bf a}_2^{\rm T} \sqrt{{\sf M}_0^e} ) (\pp_x \varphi (\x_e)  \sqrt{{\sf M}_0^e} {\bf a}_1 +  \pp_y \varphi (\x_e)  \sqrt{{\sf M}_0^e} {\bf a}_2 \right) \\
& = \left \| \pp_x \varphi (\x_e)  \sqrt{{\sf M}_0^e} {\bf a}_1 +  \pp_y \varphi (\x_e)  \sqrt{{\sf M}_0^e} {\bf a}_2  \right \|^2 \geq 0.
  \end{aligned}
\end{equation*}
So $\widetilde{\sf M}^e$ is positive semi-definite, which indicates that ${\sf M}^e({\bm \Xi})$ is positive semi-definite if ${\bm \Xi} \in \mathcal{S}_{ad}^h$. Next, we show that ${\sf K}({\bm \Xi})$ is positive definite,  we only need to show that $ \sf K_0 ({\bf \Xi})$ is positive definite, which follows the positive definiteness of the standard stiffness matrix in the finite element method. Indeed, for $\forall \bm{\alpha} \in \mathbb{R}^N,$ 
\begin{equation*}
  \bm{\alpha}^{\rm T} {\sf K_0} \bm{\alpha} = \sum_{e = 1}^M \int_{\tau_e} \Big|\sum_{i=1}^N \alpha_i \nabla \psi_i \Big|^2 \det F_e \dd \X 
  \geq c_b \int_{\Omega} \Big| \nabla \left( \sum_{i=1}^N   \alpha_i  \psi_i \right) \Big|^2 \geq 0,
\end{equation*}
where $c_b$ is defined by $c_b  = \min_{e \in \{1, 2, \ldots M \}} \det F_e$, the the equality holds only if $\alpha_i = 0$, $i = 1, 2, \ldots N$. Since for  ${\bm \Xi} \in \mathcal{S}_{ad}^h $, ${\sf M}$ is positive semi-definite, ${\sf K}$ is positive definite, we can conclude that ${\sf K}$ is positive define for $\nu > 0$. Noticed that $\det {\sf M} = 0$, so it is important to have non-zero $\nu > 0$ to guarantee the positive definiteness of ${\sf D}$.

\section{A failed example}
As mentioned previously, a pure Lagrangian calculation is sensitive to the choice of $\varphi_0$. This problem is somehow easy to deal with for the phase-field model, as it is nature to choose $\varphi_0 \in [-1, 1]$. In this appendix, we consider an extremely example by taking
\begin{equation}\label{Initial_APP}
\varphi_0(X, Y) = 2.5 (X^2 - 1)(Y^2 - 1) - 1, \quad (X, Y) \in [-1, 1]^2
\end{equation}
The boundary condition are same to section 4.2.  
 
Fig. \ref{Fail} (a) - (d) show the numerical solutions and computed meshes by our Lagrangian scheme for $\epsilon = 10^{-3}$ and $\nu = 10$ at various time. Although the mesh points can be concentrated at the thin interface, the dynamics of Lagrangian calculation is quite different with Eulerian approach, as shown in Fig. \ref{Fail} (e), and fail to get the right equilibrium. With Eulerian method, due to the discrete maximum principle, the numerical solutions will be in $[-1, 1]$ after one iteration ($t = 10^{-2}$), Then the bubble will deform into a circular bubble and shrink as in Fig. \ref{Shirkage}. But in the Lagrangian approach, since the value at each mesh point cannot be changed, the only way to minimize the total energy is to minimize the size of the region with $\varphi > 1$, and the flow map will tend to be singular at $(0, 0)$, which results in a poor mesh quality at the later stage of the Lagrangian calculations. 

This example illustrated the importance of a suitable $\varphi_0$. For general problems, we can use Eulerian approaches to obtain a proper $\varphi_0$, or combine the Eulerian methods with Lagrangian methods in the simulation to improve the robustness of the numerical scheme.
   \begin{figure}[!ht]
  \includegraphics[width = \linewidth]{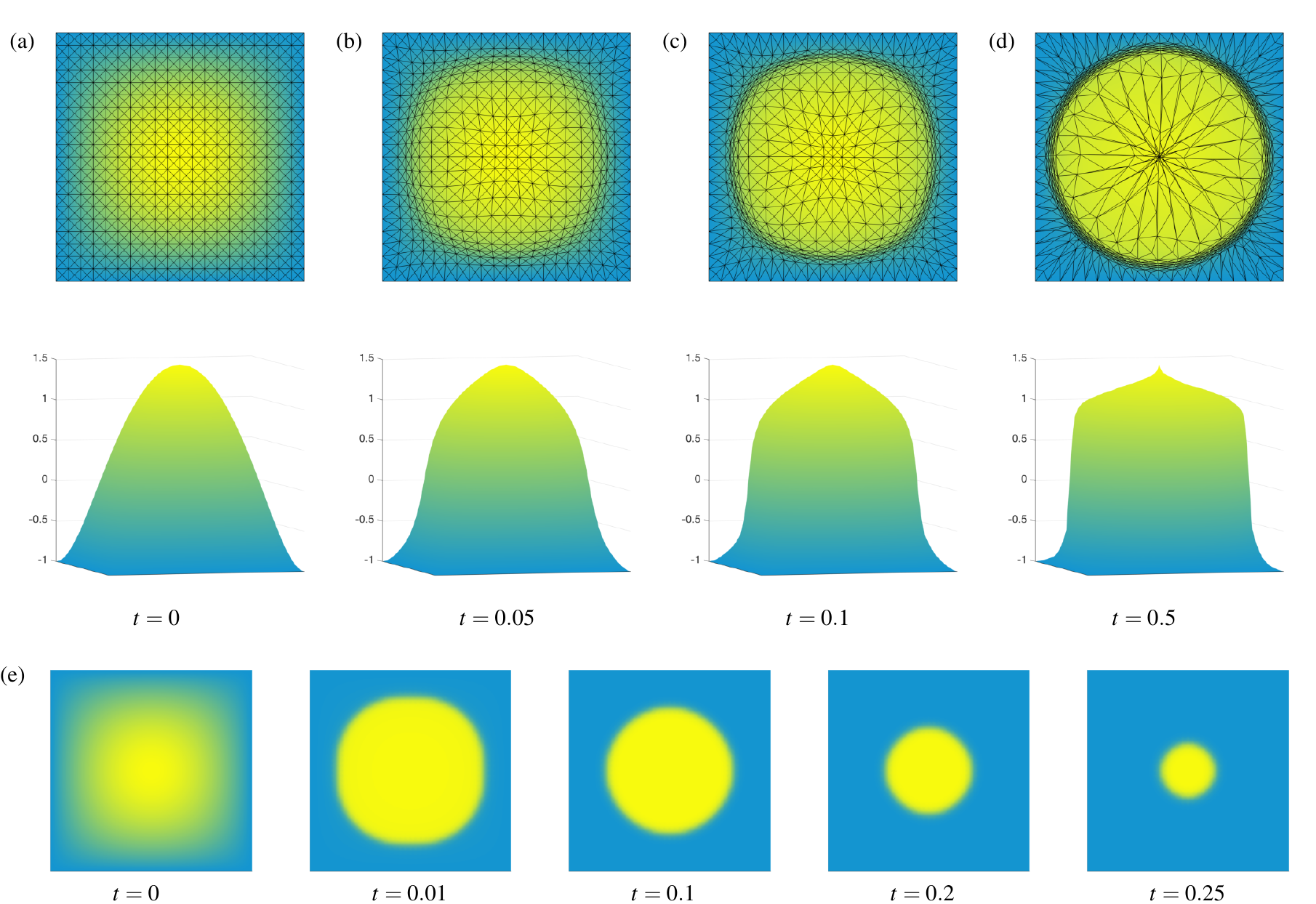} 
  \caption{(a) - (d) The meshes and computed solutions by the Lagrangian scheme at various time for the initial condition (\ref{Initial_APP}) for $\epsilon^2 = 10^{-3}$ and $\nu = 10$ [Uniform mesh. $M = 1600$]. (e) Numerical solutions for the initial condition (\ref{Initial_APP}) by an Eulerian method in a uniform mesh ($M = 64000$). The Eulerian solver used here is the one described in the section 4.}\label{Fail}
   \end{figure}

\section*{Acknowledgement}
The authors acknowledge the partial support of NSF (Grant DMS-1759536). 
Y. Wang would also like to thank Department of Applied Mathematics at Illinois Institute of Technology for their generous support and for a stimulating environment.

\bibliography{PF}

\end{document}